\documentclass[11pt,centertags,reqno]{amsart}

\usepackage[foot]{amsaddr}
\usepackage{latexsym}
\usepackage[english]{babel}
\usepackage[T1]{fontenc}
\usepackage[numbers]{natbib}
\usepackage{amssymb}
\usepackage{fancyhdr}
\usepackage{url}
\usepackage{hyperref}
\usepackage{verbatim}
\usepackage{leftidx}
\usepackage{color,graphicx}
\usepackage{pdfpages}

\usepackage{amsmath}
\usepackage{algorithm}
\usepackage[noend]{algpseudocode}

\makeatletter
\def\BState{\State\hskip-\ALG@thistlm}
\makeatother

\usepackage{mathrsfs, mathtools}
\usepackage{stmaryrd}
\usepackage{nicefrac}
\usepackage{marvosym}
\mathtoolsset{showonlyrefs}

\usepackage{lscape}
\usepackage{subfig}
\usepackage{graphicx}

\numberwithin{equation}{section} \makeatletter
\renewcommand{\subsection}{\@startsection
	{subsection}{2}{0mm}{\baselineskip}{-0.25cm}
	{\normalfont\normalsize\bf}} \makeatother

\newtheorem{theorem}{Theorem}[section]
\newtheorem{lemma}[theorem]{Lemma}

\newtheorem{proposition}[theorem]{Proposition}
\newtheorem{ass}[theorem]{Assumption}
\newtheorem{algo}[theorem]{Algorithm}

\theoremstyle{remark}
\newtheorem{remark}[theorem]{Remark}
\newtheorem{example}[theorem]{Example}

\def \F {\mathcal F}

\def \L {\mathcal L}

\def \S {\mathcal S}

\def \P {\mathbf P}

\def \I {{\mathbf 1}}

\def \R {\mathbb R}
\def \bF {\mathbb F}

\def \bE {\mathbb E}
\def \bN {\mathbb N}

\newcommand{\ud}{\mathrm d}

\newcommand{\var}{\operatorname{var}}

\makeatletter
\def\thickhline{%
	\noalign{\ifnum0=`}\fi\hrule \@height \thickarrayrulewidth \futurelet
	\reserved@a\@xthickhline}
\def\@xthickhline{\ifx\reserved@a\thickhline
	\vskip\doublerulesep
	\vskip-\thickarrayrulewidth
	\fi
	\ifnum0=`{\fi}}
\makeatother

\newlength{\thickarrayrulewidth}
\begin{document}

	\author[R.~Frey]{R\"udiger Frey}\address{R\"udiger Frey, Institute for Statistics and Mathematics, Vienna University of Economics and Business, Welthandelsplatz, 1, 1020 Vienna, Austria }\email{rfrey@wu.ac.at}
	
	\author[V.~K\"ock]{Verena K\"ock}\address{Verena K\"ock, Institute for Statistics and Mathematics, Vienna University of Economics and Business, Welthandelsplatz, 1, 1020 Vienna, Austria }\email{verena.koeck@wu.ac.at}

	\title[Convergence of the Deep Splitting Algorithm for PIDEs]{Convergence Analysis Analysis of the Deep Splitting Scheme: the Case of Partial Integro Differential Equations and the associated FBSDEs with  Jumps}

	\date{version from \today}
	
\begin{abstract}

High-dimensional parabolic partial integro-differential equations (PIDEs) appear in many applications in  insurance and  finance. Existing numerical methods suffer from the curse of dimensionality or provide solutions only for a given  space-time point.  This gave rise  to a  growing literature on   deep learning based methods for solving partial differential equations; results for integro-differential equations on the other hand are scarce. In this paper we consider an extension of the deep splitting scheme due to \citet{beck2021deep} and \citet{germain2022approximation} to PIDEs. Our main contribution is a convergence analysis of the scheme.  Moreover we discuss several test case studies to show the viability of our approach.
		
\end{abstract}

	\maketitle


\section{Introduction}

In this paper we study  deep-learning based numerical schemes for   solving   the following system of forward and backward stochastic differential equations with jumps (FBSDEJ) with unknown variables $(\mathcal X_t,Y_t,Z_t,U_t)_{0 \le t \le T}$:
\begin{equation} \label{eq:FBSDE}
\begin{cases}
\mathcal X_t&=\mathcal{ X}_0+\int_{0}^{t} b(\mathcal X_s) \, \ud s + \int_{0}^{t} \sigma(\mathcal X_{s-})  \, \ud W_s + \int_{0}^{t} \int_{\R^d} \gamma^X(\mathcal X_{s-},z) \,  \widetilde{J}(\ud s, \ud z) \, ,  \\
Y_t&=g(\mathcal X_T)-\int_{t}^{T} f(\Theta_s)  \, \ud s - \int_{t}^{T} Z_s  \, \ud W_s - \int_{t}^{T} \int_{\R^d} U_s(z)  \,  \widetilde{J}(\ud s, \ud z) \, ,  \\
\Gamma_t&=\int_{\R^d} U_t(z) \rho(z) \,  \nu (\ud z) \, .
\end{cases}
\end{equation}
Here  $\mathcal{X}$ is a $d$-dimensional process with  dynamics given in terms of functions $b \colon \R^d \to \R^d$, $\sigma \colon \R^d \to  \R^{d \times d}$, $\gamma^X(x,z) \colon \R^d \times \R^d \to \R^d$ for a $d$-dimensional Brownian motion $W$ and  an independent compensated Poisson measure $\widetilde J(\ud t, \ud z)=J(\ud t, \ud z)- \nu(\ud z) \, \ud t$ for a finite measure $\nu $ on $\R^d$. Moreover,  $\rho \colon \R^d \to \R$ is a bounded measurable function and  $\Theta_s=(s,\mathcal X_s,Y_s,Z_s,\Gamma_s)$.  The (generally nonlinear) function $f \colon [0,T]\times \R^d \times \R \times \R^d \times \R \to \R$ is the driver of the backward equation in \eqref{eq:FBSDE}, and $g \colon \R^d \to\R$ gives the terminal condition.

FBSDEJ systems of the form   \eqref{eq:FBSDE} arise in hedging or utility maximization problems in mathematical finance; see for instance  \citet{eyraud2005backward}  or  \citet{becherer2006bounded}.
A further important  motivation for studying numerical methods for the system~\eqref{eq:FBSDE} is that  under fairly weak  conditions the FBSDEJ provides an alternative representation for a large class of partial integro-differential equations (PIDEs) of the form
\begin{align}\label{eq:PIDE}
	\begin{cases}
	u_t(t,x)+\L u (t,x)  = f\big(t,x,u(t,x),\sigma(x)^\top D_x u(t,x),\mathcal{I}[u](t,x)\big)    & \text{on } [0,T) \times \R^d \, , \\
	u(T,x)=g(x) &    \text{on } \R^d  \, .
	\end{cases}
\end{align}
Here   $D_x u$  is the gradient of $u$ with respect to the space variable,  and
\begin{align}
\L u(t,x) &:=  b(x) \cdot D_x u (t,x)  +\frac{1}{2} \sum_{i,j=1}^{d} (\sigma \sigma^\top)_{ij}(x)  u_{x_i x_j}(t,x) \label{eq:sec1_generatorX}\\ & \qquad + \int_{\R^d} \big [u(t,x+\gamma^X(x,z))-u(t,x)- D_x u (t,x) \cdot \gamma^X(x,z)\big ]  \,  \nu(\ud z) \, , \\
\mathcal{I}[u](t,x)&:=\int_{\R^d} \big [ u(t,x+\gamma^X(x,z))-u(t,x)\big]\rho(z)  \,  \nu(\ud z) \,. \label{eq:sec1_generatorI}
\end{align}
PIDEs of the form~\eqref{eq:PIDE} appear naturally in many  finance and insurance related problems and in certain stochastic control problems for jump diffusions. 

Existing numerical  methods for PIDEs typically suffer  from the curse of dimensionality, as it is the case for finite difference and finite element methods, or they  only provide a solution for a single fixed time-space point $(t,x)$, as it is the case for schemes based on Monte-Carlo methods.
Starting with the seminal papers  \citet{han2017overcoming} and  \citet{bib:e-han-jentzen-17}, this has led to a large  literature on   machine learning based numerical methods  for parabolic  PDEs without  non-local terms.    Many approaches are  based on the  FBSDE representation of the equations.
To begin with, \citet{hure2020deep}
estimate the solution and its gradient simultaneously by backward induction through sequential minimizations of suitable loss functions; moreover, they  provide convergence results for their method. The paper of \citet{beck2021deep} uses a different discretization method, called \emph{deep splitting} or \emph{DS},  that computes the unknown gradient of the solution by automatic differentiation, which  reduces  the complexity  of the  network approximation. For linear PDEs this method simplifies   to the  global regression approach of \citet{beck2021solving}.
\citet{pham2021neural} combine the ideas of  \cite{hure2020deep}  and \cite{beck2021deep} to introduce a neural network scheme for fully nonlinear PDEs. Finally, \citet{germain2022approximation} extend the method in \cite{hure2020deep} and  they provide a convergence analysis that covers many schemes for semilinear PDEs including.

Work on  deep learning methods for PIDEs or FBSDEs with jumps (FBSDEJs) on the other hand is scarce. In the present paper we study therefore an extension of the deep splitting scheme for FBSDEJs and PIDEs. Our main mathematical result is a detailed analysis of the convergence properties of the scheme. Our proof follows the approach  of \cite{germain2022approximation} but some important changes have to be incorporated due to the non-local character of the generator \eqref{eq:sec1_generatorX} and the integral term  in the driver $f$. We moreover consider  the special  case where the driver is of the form $f = f(t,x)$ (linear PIDEs); in that case tighter bounds on the approximation error can be given.
Finally, we present several numerical case studies   to show the viability of the approach.
Further examples, extensions to boundary value problems and details of the algorithmic implementation are provided in the companion paper \citet{frey2021deep}.

We continue with a brief discussion of the existing contributions on deep learning for PIDEs.
\citet{castro2021deep}  extends the method of \cite{hure2020deep} to FBSDEJs and proves convergence of the generalized scheme.  From a numerical viewpoint his  method  is   quite  involved, since one  needs to approximate the solution $Y$ and the  the integrands $Z$ and $U_t(\cdot)$ in \eqref{eq:FBSDE} by three separate networks. \citet{bib:gonon2021deep} propose deep learning based methods for linear PIDEs. Very recently  \citet{bib:boussange2022deep} proposed an extension of the deep splitting schemes to PIDEs of the form \eqref{eq:PIDE} with Neumann boundary condition. However, they do not study  the convergence properties of their scheme, and they assume that the driver  $f$ is independent of $D_x u$, which excludes many relevant control problems. Finally, \citet{bib:neufeld2022multilevel} consider multilevel Picard approximation for semilinear PIDEs and they provide a complexity analysis for their algorithm  (again in the case where $f$ does not depend on $D_x u$).

The paper is structured as follows. In Section \ref{sec:theFBSystemJ} we describe the problem framework.  In Section~\ref{sec:algo} we summarize the work of \citet{bouchard2008discrete} on the Euler scheme for FBSDEs and we introduce the FBSDEJ-version of the  deep splitting  algorithm. The error bound for the DS  algorithm is discussed in Sections \ref{sec:errorbound} and \ref{sec:proof}. In the final Section \ref{sec:numerical} we present several numerical case studies.

\section{The Forward Backward System with jumps} \label{sec:theFBSystemJ}

\subsection{Notation.}

We fix a probability space $(\Omega,\F, \P)$, a time horizon $T$ and a right continuous  filtration $\bF$.
  Let $|x|_2$ denote the Euclidian norm for $x \in \R^d$,   $\|X\|_{2p}:=(\bE|X|_2^{2p})^{\frac{1}{2p}}$, $p \ge 1$ the $L^{2p}$-norm for a random vector $X \in \R^d$ and $|A|_2:=\sup \{|Ax|_2: x \in \R^d, \, |x|_2=1 \}$ for $A \in \R^{d\times d}$.  For sake of simplicity for random vectors $Z_1, Z_2 \in \R^d$, we use the convention
$\mathbb{E}|Z_1-Z_2|_2^2:=\mathbb{E}[|Z_1-Z_2|_2^2] \, .$ Given a time point  $t_i \in [0,T]$ (usually an element of a time grid) the operator $\mathbb{E}_i$ will denote the conditional expectation with respect to $\mathcal{F}_{t_i}$, that is  for a generic  $X \in L^1(\Omega,\F, \P)$ we let
$\mathbb{E}_i[X]:=\mathbb{E}[X|\mathcal{F}_{t_i}]$. Moreover, we denote by $\var_i$ the conditional variance with respect to $\mathcal{F}_{t_i}$, that is for $X\in L^2(\Omega,\F, \P)$,
\begin{equation}\label{eq:cond-var}
\var_i(X) =  \mathbb{E}_i \,\big | X - \mathbb{E}_i[X] \, \big |^2  = \mathbb{E}_i[X^2] - \mathbb{E}_i[X]^2\,.
\end{equation}

Given $s\le t$ and $p \ge 2$ we introduce  the following  spaces of stochastic processes (see also \cite{bouchard2008discrete}). First,
$\S^p_{[s,t]}$ is the set of all adapted c\`adl\`ag processes $Y$ such that
$\|Y\|_{\S^p_{[s,t]}} := \mathbb E \big[\sup_{t \in [s,t]} |Y_t|^p \big]^{1/p} < \infty $. Second,
$L^p_{W,[s,t]}$ is the set of all progressively measurable $\R^d$-valued   processes $Z$ with
$$ \|Z\|_{L^p_{W,[s,t]}} := \mathbb E \Big[\Big( \int_{s}^t |Z_t|_2^2  \,  \ud t \Big)^{p/2} \Big]^{1/p} < \infty.$$
Third, denote by $\mathcal{P}$ the $\sigma$-algebra of $\mathbb{F}$-predictable subsets of $ \Omega \times [0,T]$. Then
$L^p_{\nu,[s,t]}$ is the set of all maps $U: \Omega \times [0,T] \times \R^d \to \R$ that are $\mathcal{ P} \otimes \mathcal{B}(\R^d)$ measurable with
$$\|U\|_{L^p_{\nu, [s,t]}} := \mathbb E \Big[\int_{s}^t \int_{\R^d} |U_t(y)|^p \, \nu (\ud y)  \,   \ud t \Big]^{1/p} < \infty \, .$$
The space $\mathcal{B}^p_{[s,t]}:=\S^p_{[s,t]}\times L^p_{W,[s,t]} \times  	L^p_{\nu,[s,t]}$ is finally endowed with the norm
$$ \|(Y,Z,U)\|_{\mathcal{B}^p_{[s,t]}} =\Big(\|Y\|^p_{\S^p_{[s,t]}}  +  \|Z\|^p_{L^p_{W,[s,t]}} + \|U\|^p_{L^p_{\nu, [s,t]}} \Big)^{1/p} \, .$$
Whenever $(s, t) = (0, T )$ we omit the subscript $[s, t]$ in these notations.


\subsection{The  FBSDEJ}

We assume that $(\Omega,\F, \P)$  supports  a $d$-dimensional Brownian motion $W$   and a Poisson random measure $J(\ud t, \ud z)$ on $[0,T]\times \R^d$. The compensator of $J$ is given by  $\nu(\ud z)  \,  \ud t$ for a finite measure $\nu$ on $\R^d$ with $\nu(\{0\})=0$. The compensated measure is denoted as
$$\widetilde J(\ud t, \ud z)=J(\ud t, \ud z)- \nu(\ud z)   \,  \ud t \, ,$$
such that for every measurable set  $A \in \R^d$ the process $\widetilde{J}([0,t],A)$, $t \ge 0$, is a martingale.

Consider measurable functions $\mu: [0,T] \times \R^d \to \R^d $, $\sigma: [0,T] \times \R^d \to \R^{d\times d}$,  $\gamma^X: [0,T] \times \R^d \times \R^d \to \R^d$ and a bounded measurable function $\rho \colon \R^d \to \R$. In this paper we consider the following FBSDE system with jumps.
\begin{align}
\mathcal X_t&=\mathcal{ X}_0+\int_{0}^{t} b(\mathcal X_s)  \, \ud s + \int_{0}^{t} \sigma(\mathcal X_{s-})  \, \ud W_s + \int_{0}^{t} \int_{\R^d} \gamma^X(\mathcal X_{s-},z) \widetilde{J}(\ud s, \ud z) \, ,  \label{eq:FBSDEJ_X}\\
Y_t&=g(\mathcal X_T)-\int_{t}^{T} f(\Theta_s)  \, \ud s - \int_{t}^{T} Z_s \,  \ud W_s - \int_{t}^{T} \int_{\R^d} U_s(z)  \, \widetilde{J}(\ud s, \ud z) \, ,  \label{eq:FBSDEJ_Y} \\
\Gamma_t&=\int_{\R^d} U_t(z) \rho(z)  \, \nu (\ud z)  \label{eq:FBSDEJ_Gamma}\,.
\end{align}

Throughout  we  make the following assumption on the coefficients of this  FBSDEJ system.

\begin{ass} \label{ass:conditions}   There exists a universal constant $K > 0 $ such that
 \begin{enumerate}
 	\item The functions $b: \R^d \to \R$, $\sigma: \R^d \to \R^{d \times d }$ are $K $-Lipschitz continuous.
 	\item The map $\gamma^X: \R^d \times \R^d \to \R^d$ is measurable, uniformly bounded and uniformly $K $-Lipschitz, i.e.  $\sup_{z \in \R^d} |\gamma^X(0,z)|\leq K$
	and $\sup_{z \in \R^d} |\gamma^X(x,z)-\gamma^X(x',z)| \leq K |x-x'|, \forall x,x' \in \R^d$
 	\item For each $t,t',y,y',w,w' \in \R$ and $x,x',z,z' \in \R^d$ the map $f \colon [0,T]\times \R^d \times \R \times \R^d \times \R \to \R$ is $[f]_L $-Lipschitz with $[f]_L<K$, that is
 	$|f(t,x,y,z,w)-f(t',x',y',z',w')| \le [f]_L\big(|t-t'|^{1/2}+|x-x'|_2+|y-y'|+|z-z'|_2+|w-w'|\big)$.
 	\item The function $g: \R^d \to \R$ is $K$-Lipschitz continuous.
 	\item For each $z \in \R^d$, the map $x \mapsto \gamma^X(x,z)$ admits a Jacobian matrix $D_x \gamma^X(x,z)$ such that the function $a(\cdot \, ;z)$,  $(x,\xi) \in \R^d \times \R^d \mapsto a(x,\xi;z)=\xi^\top (D_x \gamma^X(x,z)+I_d) \xi$ satisfies either  $a(x,\xi;z) \geq |\xi|^2 {K}^{- 1}$ for all  $(x,\xi) \in \R^d \times \R^d$ or $a(x,\xi;z)\leq -|\xi|^2 {K}^{-1}$ for all  $(x,\xi) \in \R^d \times \R^d$.
 \end{enumerate}
\end{ass}
The existence and uniqueness of a solution $\mathcal{X}$ to \eqref{eq:FBSDEJ_X} is guaranteed under conditions (1)-(2), and in order to ensure existence and uniqueness of a solution $Y$ to \eqref{eq:FBSDEJ_Y} we have to assume (3)-(4), see  \citet{tang1994necessary} for details. Condition~(5) on the other hand is more technical and implies that the matrix $D_x \gamma^X(x,z)+I_d$ is invertible with inverse bounded by $K^{\, {-}1}$. 
\citet{bouchard2008discrete} show that (5) is required to ensure the proper convergence order of the Euler scheme for \eqref{eq:FBSDEJ_Y}, see   Section \ref{sec:disc} below.

The standard estimates for solutions of FBSDEJs (see for instance \cite{bouchard2008discrete}) imply that under Assumption~\ref{ass:conditions}~(1)--(4), \marginpar{check} there is some $C>0$ such that
\begin{align*}
\|(\mathcal X,Y,Z,U)\|^p_{\S^p \times \mathcal{B}^p} &\leq C (1+\bE|\mathcal{ X}_0|_2^p) \, , \\
\mathbb{E} \Big[\sup_{s\le u \leq t}|\mathcal X_u-\mathcal X_s|_2^p\Big] &\leq C(1+\bE|\mathcal{ X}_0|_2^p)|t-s| \, , \\
\mathbb{E} \Big[\sup_{s\le u \leq t}|Y_u-Y_s|^p\Big] &\leq C \Big [ (1+\bE|\mathcal{ X}_0|_2^p)|t-s|^p + \|Z\|^p_{L^p_{W,[s,t]}}+\|U\|^p_{L^p_{\nu,[s,t]}} \Big ] \, .
\end{align*}

Next we discuss the relation between the FBSDEJ system \eqref{eq:FBSDEJ_X}-\eqref{eq:FBSDEJ_Gamma} and the PIDE
\eqref{eq:PIDE}. Suppose that $u \in \mathcal{ C}^{1,2}([0,T],\R^d)$ is a classical solution  of \eqref{eq:PIDE}. Then we obtain by applying the It\^o's formula to $Y_t = u(t, \mathcal{X}_t) $ that the forward process $\mathcal{X}$ and the triple   $(Y_t,Z_t,U_t)$ with
\begin{align} \label{eq:classical-sol}
	Y_s=u(s,\mathcal X_s) \, , \quad Z_s=\sigma(\mathcal X_s)^\top D_x u(x,\mathcal X_s) \, , \quad U_s(z)= u(s,\mathcal X_s+\gamma^X(\mathcal  X_s,z))-u(s,\mathcal  X_s) \, .
\end{align}
solves the FBSDEJ system. More generally, it is well known (see e.g. \cite{barles1997backward}) that, under mild assumptions on the coefficients given in Assumptions~\ref{ass:conditions}, the (viscosity) solution $u$ can be related to the component $Y$ of the solution to  \eqref{eq:FBSDEJ_X}-\eqref{eq:FBSDEJ_Gamma} in terms of the unknown variables $(\mathcal X_t,Y_t,Z_t,U_t)$ in the sense that $Y_t=u(t,\mathcal X_t)$.


\section{The DS-algorithm} \label{sec:algo}

In this section  we introduce the deep splitting method of \citet{beck2021deep}. We follow  \citet{germain2022approximation}, who discuss  the algorithm in an FBSDE context.

\subsection{Discrete-time approximation of FBSDEJs} \label{sec:disc}
To motivate the algorithm  we recall the Euler approximation for FBSDEJ systems from \citet{bouchard2008discrete}.
Let $\pi$ be a (for simplicity)  constant partition $\{t_0=0<t_1<\dots<t_N=T\}$ with modulus $|\pi| := \Delta t := T/N$.
The Euler scheme for \eqref{eq:FBSDEJ_X} takes the form
\begin{align}\label{eq:Eulerscheme}
\begin{cases}
	X_0^\pi:=\mathcal{X}_0,  \\
	X_{t_{i+1}}^\pi := X^\pi_{t_i}+b(X^\pi_{t_i})\Delta t_{i}+\sigma(X^\pi_{t_i}) \Delta W_{i}+ \int_{t_i}^{t_{i+1}} \int_{\R^d} \gamma^X(X^\pi_{t_i},z)  \,  \widetilde{J}(\ud t,\ud z) \, .
	\end{cases}
\end{align}
where  $\Delta W_{i}=W_{t_{i+1}}-W_{t_i}$.
For convenience we will always write $X_i=X^\pi_{t_i}$ and use $X_i$ as approximation for $\mathcal{ X}_t$ for each $t$ in the interval $[t_i,t_{i+1})$.
We define the continuous component $X^c$ and the jump component $X^J$  as
\begin{align}
	 X^c_{i+1}:=X_{i}+b(X_{i})\Delta t_{}+\sigma(X_{i}) \Delta W_{i}, \quad X_{i+1}^J:=\int_{t_i}^{t_{i+1}} \int_{\R^d} \gamma^X(X_{i},z)  \,  \widetilde{J}(\ud t ,\ud z) \, .
\end{align}

\begin{lemma}[Euler scheme for $\mathcal{X}$] \label{lemma:forward-Euler}
There is  a constant $C>0$ such that
\begin{equation} \label{eq:Eulererror}
	\max_{i=1,2,\dots, N} \mathbb{E} \Big[ \sup_{t \in [t_i,t_{i+1})} |\mathcal{ X}_t-X_{i}|_2^2\Big] \le C |\pi| \,.
\end{equation}
The standard estimates for $X_i$ are
$
\|X_i\|_{2p} \le C(1+\|\mathcal{ X}_0\|_{2p})$ and $\|X_{i+1}-X_i\|_{2p} \le  C(1+\|\mathcal{ X}_0\|_{2p})\sqrt{\Delta t} \, .$
\end{lemma}
We use the \emph{explicit backward Euler scheme} (see e.g. \citet{bouchard2008discrete}) to approximate the backward variables of the FBSDEJ. Under this scheme the  triplet $(Y,Z,\Gamma)$   is approximated by processes $(\bar V, \bar Z,\bar \Gamma)$ that are defined by the following backward recursion. First,  $\bar V_N:=g({X}_N)$. Next, with $\Delta M_{i}= \int_{t_i}^{t_{i+1}}\int_{\R^d} \rho(z) \widetilde{J}(\ud s, \ud z)$ we have on each interval $[t_i,t_{i+1})$, $0 \le i \le N-1$
\begin{align} \label{eq:expl_euler_scheme1}
\bar V_i &:= \mathbb{E}_i \Big[ \bar V_{i+1} - f(t_i,X_{i},\bar{V}_{i+1},\bar Z_i, \bar \Gamma_i)\Delta t \Big] \, , \\
\bar Z_i&:= \mathbb{E}_i \Big[ \bar V_{i+1} \frac{\Delta W_{i}}{\Delta t} \Big] \, , \quad
\bar \Gamma_i := \mathbb{E}_i \Big[ \bar V_{i+1} \frac{\Delta M_{i}}{\Delta t} \Big]\,.  \label{eq:expl_euler_scheme2}
\end{align}
 We let $(\bar Y_t,\bar Z_t,\bar \Gamma_t):=(\bar Y_i,\bar Z_i,\bar \Gamma_i)$ for $t$ in the interval $[t_i,t_{i+1})$. The convergence properties of the backward Euler scheme are closely related to the $L^2$-\emph{regularity errors} $\epsilon^Z(\pi)$  and $\epsilon^\Gamma(\pi)$ of the solution  to the FBSDEJ.  These quantities are given by
\begin{align} \label{eq:L2error}
 \epsilon^Z(\pi)&:= ||Z-\widetilde{Z}||^2_{L^{2}_W}=  \sum_{i=1}^{N-1}  \bE \Big [ \int_{t_i}^{t_{i+1}}|Z_t-   \widetilde Z_{t_i}|_2^2  \,  \ud  t \Big ]  ,  \text{ where } \widetilde Z_{t}:=\frac{1}{\Delta t} \mathbb{E}_i\Big[\int_{t_i}^{t_{i+1}} Z_s  \,  \ud s \Big]  \text{ for }t \in [t_i,t_{i+1}) \, , \\
\epsilon^\Gamma(\pi)&:= ||\Gamma-\widetilde{\Gamma}||^2_{L^{2}_W}= \sum_{i=1}^{N-1} \bE \Big [ \int_{t_i}^{t_{i+1}}|\Gamma_t-   \widetilde \Gamma_{t_i}|^2  \,  \ud  t \Big ]   ,  \text{ where  } \widetilde \Gamma_{t}:=\frac{1}{\Delta t} \mathbb{E}_i\Big[\int_{t_i}^{t_{i+1}} \Gamma_s\, \ud s \Big] \text{ for }t \in [t_i,t_{i+1}) \, .
\end{align}

The following results on the  convergence order of the explicit backward Euler scheme from \citet{bouchard2008discrete} are a key tool in our analysis.
\begin{proposition} \label{prop:backward-Euler} Under Assumption~\ref{ass:conditions}  (1)--(4)  the approximation error of the backward Euler scheme satisfies
%
%
\begin{align}
\label{eq:time_dis_error}
\max_{i \in \{0,1,\dots,N\}} \sup_{t \in [t_i,t_{i+1}]} \mathbb{E} |Y_{t}-\bar V_i|^2&+\sum_{i=0}^{N-1}\mathbb{E} \Big[ \int_{t_i}^{t_{i+1}}|Z_s-\bar Z_i|_2^2 \, \ud s \Big ] +\sum_{i=0}^{N-1} \mathbb{E} \Big[ \int_{t_i}^{t_{i+1}}|\Gamma_s-\bar \Gamma_i|^2 \, \ud s \Big ] \\ &
\leq C \big( |\pi| + \epsilon^Z(\pi)+ \epsilon^\Gamma(\pi) \big) \, .
\end{align}
The $L^2$-regularity error $\epsilon^\Gamma(\pi)$ is of order  $O(|\pi|)$. If moreover Assumption~\ref{ass:conditions} (5) holds, $\epsilon^\Gamma(\pi)= O(|\pi|) $.
\end{proposition}

\begin{remark}
	Note that Assumption~\ref{ass:conditions}  (1)--(4) ensure that $\epsilon^\Gamma(\pi)=||\Gamma-\widetilde{\Gamma}||^2_{L^{2}_W} \le Cn^{-1}$ and $\epsilon^Z(\pi)= ||Z-\widetilde{Z}||^2_{L^{2}_W}  \le Cn^{-1+\epsilon}$ for any $\epsilon>0$. Under the additional Assumption~\ref{ass:conditions} (5)  on the invertibility of $D_x \gamma^X(x,z)+I_d$ the previous inequality is  true  even for $\epsilon=0$, see Proposition~2.1 and Theorem~2.1 in \cite{bouchard2008discrete}.
\end{remark}

\subsection{The algorithm}

We fix a class $\mathcal{N}$ of $C^1$ functions  $\mathcal{U}:\R^d \to \R$ that are given in terms of neural networks with fixed structure.   The basic idea of the DS algorithm is to use the  $L^2$-minimality of conditional expectation  to rewrite the conditional expectation in \eqref{eq:expl_euler_scheme1} as a  regression problem where one projects at each $t_i$ on the set  of random variables $\mathcal{N}_{i} := \{ \mathcal{U}(X_{i}) \colon \mathcal{U} \in \mathcal{N} \}$ (and not on all of $ L^2(\Omega, \F_{t_i}, P)$). Moreover, motivated by \eqref{eq:classical-sol}, one  replaces $\bar Z_i$ by $   D_x \mathcal{\widehat U}_{i+1}(X_{i+1})$  and $ \bar \Gamma_i$ by $\mathcal{\widehat U}_{i+1}(X^c_{i+1}+\gamma^X(X_{{i}},z))-\mathcal{\widehat U}_{i+1}(X^c_{i+1})$. This leads to the following algorithm.

\begin{algo}[DS algorithm] \label{alg:1} Choose   a class $\mathcal{N}$ of $C^1$ network functions $\mathcal{U} \colon \R^d \to \R$.
 Then the  algorithm proceeds by backward induction  as follows.
\begin{enumerate}
	\item If $g \in \mathcal{N}$, let $\mathcal{\widehat U}_N=g$. Otherwise define $ \mathcal{\widehat U}_N \in \mathcal{N} $ as minimizer of the terminal loss function
$ L_N \colon \mathcal{N} \to \R, \quad  \mathcal{U} \mapsto  \mathbb{E}|g(X_N)-\mathcal{U}(X_N)|^2 .$	
	\item For $i=N-1,\dots,1,0,$  define the integral operator
	\begin{align} \label{eq:Ii}
		\mathcal{I}[\mathcal{\widehat U}_{i+1},X_i](X^{c}_{i+1}):=\int_{\R^d}   \rho(z) \big [\mathcal{\widehat U}_{i+1}(X_{i+1}^c+\gamma^X(X_{{i}},z))-\mathcal{\widehat U}_{i+1}(X_{i+1}^c)\big ] \,  \nu (\ud z) \, , \quad
	\end{align} and choose
	 $ \mathcal{\widehat U}_i $ as minimizer of the  loss function
$L_i \colon \mathcal{N} \to \R$,
	\begin{align}\label{eq:loss_fun}
	\mathcal{U} \mapsto \mathbb{E}\Big|\mathcal{\widehat U}_{i+1}(X_{i+1})-\mathcal{U}(X_{i})-
    \Delta t \, f\Big(t_{i},X_{{i+1}},\mathcal{\widehat U}_{i+1}(X_{i+1})\sigma(X_{i}) D_x \mathcal{\widehat U}_{i+1}(X_{i+1}),
      \mathcal{I}[\mathcal{\widehat U}_{i+1},X_i](X^{c}_{i+1})\Big)   \Big|^2 \quad \ \
	\end{align}
	\end{enumerate}
\end{algo}

Conditions on $\mathcal{N}$ which ensure that the minimization problems in Steps~1 and 2 possess a solution are discussed in the next section.

\begin{remark}[Numerical implementation] In order to implement the algorithm numerically, one generates a large number $M$ of paths $X^1, \dots, X^M$ of $X$ and approximates the loss function $L_i$ by
\begin{align} \label{eq:loss-fun-discrete}
\frac{1}{M} \sum_{m=1}^M & \Big | \, \mathcal{\widehat U}_{i+1}(X_{i+1}^m)-\mathcal{U}(X_{i}^m) -
    \Delta t \, f\Big(t_{i},X_{{i+1}}^m,\mathcal{\widehat U}_{i+1}(X_{i+1}^m),\sigma(X_{i}^m) D_x \mathcal{\widehat U}_{i+1}(X_{i+1}^m), \mathcal{I}[\mathcal{\widehat U}_{i+1},X_i^m](X^{m,c}_{i+1}) \Big) \,\Big |^2.
\end{align}
To minimize \eqref{eq:loss-fun-discrete}  the network representing $\mathcal{U}$ is  trained using stochastic gradient descent,  see for instance \cite{frey2021deep} for details.    Note that the evaluation of the integral term
$$\int_{\R^d}[\mathcal{\widehat U}_{i+1}(X^{m,c}_{i+1}+\gamma^X(X_{{i}}^m,z))-\mathcal{\widehat U}_{i+1}(X^{m,c}_{i+1})] \, \nu (\ud z)$$
in \eqref{eq:loss-fun-discrete} is also a high-dimensional problem. However this problem may be tackled with Monte Carlo approximations that can be done off-line, that is before the training procedure for the network.
\end{remark}

\begin{remark}[The linear case] \label{remark:linearcase}
Suppose that the FBSDEJ is \emph{linear}, that is $f = f(t,x)$ and define
\begin{equation} \label{eq:def-Hi}
H_i := g(X_N) - \Delta t\sum_{n=i}^{N-1} f(t_n,X_n) \text{ for }\, i=0,1,\dots,N-1.
\end{equation}
In the linear case  the backward Euler reduces to the simpler expression
$\bar V_i = \mathbb{E}_i \big [H_i \big]$ (by the law of iterated conditional expectations). Using the $L^2$-minimality of conditional expectations we therefore define $\mathcal{\widehat U}_i \in \mathcal{N} $ as minimizer of the loss function
\begin{equation} \label{eq:loss-function-linear}
L_i^{\text{lin}}(\mathcal{U})=  \mathbb{E} \big |  \mathcal{U}(X_{i}) - H_i \big |^2 \,.
\end{equation}
Minimizing the loss function $L_i^{\text{lin}}$ corresponds to the algorithm proposed by \citet{beck2021solving} for  linear parabolic PDEs. Note that in the linear case there is  no need to work with networks of $C^1$ functions.
\end{remark}

\section{Bound on the approximation error } \label{sec:errorbound}

In this section we discuss error  bounds for the DS algorithm.   We assume that the class $\mathcal{N}$ of  network functions has the following properties.
\begin{ass} \label{ass:N}
 There exist constants $\gamma, \eta >0$ such that for all $\mathcal{U} \in \mathcal{N}$
\begin{align*}
|\mathcal{\widehat U}_{i}(x)-\mathcal{\widehat U}_{i}(x')| &\le \gamma( 1+ \max(|x|_2 ,|x'|_2))|x-x'|_2\; \text{ and }\;
|D_x\mathcal{\widehat U}_{i}(x)-D_x\mathcal{\widehat U}_{i}(x')|_2 \le \eta |x-x'|_2 \, .
\end{align*}
Moreover, we suppose that  the set $\{ |\mathcal{U} (0)| + |D_x \mathcal{U}(0)|_2 \colon  \mathcal{U} \in \mathcal{N} \}$ is bounded and that $\mathcal{N}$ is closed with respect to $C^1$-convergence on compact subsets of $\R^d$.
\end{ass}
In the sequel  we write $\mathcal{N}^{\gamma,\eta}$ to indicate the dependence of  $\mathcal{N}$  on these parameters. Functions from $\mathcal{N}^{\gamma,\eta}$ satisfy the quadratic growth condition
\begin{equation} \label{eq:growth}
\mathcal{U}(X)^2  \le C({\gamma,\eta})( 1 +  | x |_2^2) \text{ for some  constant $C({\gamma,\eta})$.}
 \end{equation}

 \begin{example} A  class of networks with a single hidden layer (shallow network)  that satisfy Assumption~\ref{ass:conditions} is considered in  \citet{germain2020deep}. They  consider network functions $\varphi: \R^d \to \R$,   with $m$ nodes and activation function \texttt{ReLu$_{2}$}, i.e. $\rho(x)=\max(0,x)^2$ that take the form
\begin{equation} \label{eq:networks}
\varphi(x)=\mathcal{W}_1\rho(\mathcal{W}_0 x+\beta_0 )+\beta_1    \, ,
\end{equation}
for network parameters $\theta=(\mathcal{W}_0,\beta_0,\mathcal{W}_1,\beta_1) \in \R^{m\times d} \times \R^m \times \R^{1 \times m} \times \R$ and where for $x \in \R^m$, $\rho(x) = (\rho(x_1), \dots, \rho(x_m))^\top$.
We denote for $R,\kappa \geq 1 $ by $\mathcal{N}_{d,m}^{R,\kappa}$ the set of network functions \eqref{eq:networks} with parameters $(\mathcal{W}_0,\beta_0,\mathcal{W}_1,\beta_1)$ satisfying row by row
\begin{align*}
|(\mathcal{W}_0^i,\beta_0^i/R)|_2 = \frac{1}{R}, i=1,2,\dots,m, \text{ and } |(\mathcal{W}_1,\beta_1)|_1 \le \kappa,
\end{align*}
where $|\cdot|_1$ and $|\cdot|_2$ are the $\ell_1$ and $\ell_2$ norms in Euclidean spaces. It is shown in \citet{germain2020deep} that functions $\varphi \in \mathcal{N}_{d,m}^{R,\kappa}$ and their derivatives $D_x \varphi(x)$ are Lipschitz and satisfy in particular Assumption~\ref{ass:N}. 
 \end{example}

We now show that under Assumption~\ref{ass:N}, the minimization problems in Algorithm~\ref{alg:1} do  have a solution. The Arzela Ascoli theorem implies  that $\mathcal{N}^{\gamma,\eta}$ is  compact with respect to $C^1$-convergence on compact subsets of $\R^d$. Fix now some grid  point $t_i$ and consider a sequence $\mathcal{U}^m \in \mathcal{N}^{\gamma, \eta}$ such that $L_i (\mathcal{U}^m) \to L_i^* := \inf \{L_i (\mathcal{U}) \colon \mathcal{U} \in \mathcal{N}^{\gamma, \eta}\}$ for $m \to \infty$. By compactness there is a subsequence $\mathcal{U}^{m'}$ and some $\mathcal{U}^* \in \mathcal{N}^{\gamma, \eta}$ with $\lim_{m' \to \infty} \mathcal{U}^{m'} = \mathcal{U}^*$  (with respect to ${C}^1$-convergence on compact sets). We want to show that $\mathcal{U}^*$ is a minimizer of $L_i$, i.e. $L_i^*=L_i(\mathcal{U}^*)$. To this note first note that $\mathcal{U}^{m'}(X_i)$ converges to  $\mathcal{U}^*(X_i) $ in $L^2$ for $m' \to \infty$, since
$$ \bE | \mathcal{U}^{m'}(X_i) -  \mathcal{U}^*(X_i) |^2 =  \bE \big [ (\mathcal{U}^{m'}(X_i) -  \mathcal{U}^*(X_i))^2 \I_{|X_i|_2 \le R} \big ] +   \bE \big [ (\mathcal{U}^{m'}(X_i) -  \mathcal{U}^*(X_i))^2 \I_{|X_i|_2 > R} \big ] \, .
$$
By the quadratic growth condition~\eqref{eq:growth} the second term on the right is bounded by
$$2 C({\gamma,\eta}) \bE \big[ (1 + |X_i|_2^2 ) \I_{|X_i|_2 > R} \big ] \, , $$ and this  converges to zero for $R \to \infty$ as $|X_i|_2$ is square integrable. Moreover,  for fixed $R$ the first term on the right converges to zero for $m' \to \infty$ as $\mathcal{U}^{m'} \to \mathcal{U}^*$ in $C^1$ on the compact set $\{ |x_2| \le R\}$. Second, we get from the triangle inequality that
$ L_i(\mathcal{U}^*)^{1/2} \le \|\mathcal{U}^{m'} - \mathcal{U}^*\|_2 + L_i(\mathcal{U}^{m'})^{1/2}$. Now the r.h.s. converges to $(L_i^*)^{1/2}$  for $m' \to \infty$, so that
 $L_i(\mathcal{U}^*) = L_i^*$ and $\mathcal{U}^* $ is in fact a minimizer of $L_i$.

To derive bounds on the approximation error we use the approximation result for the backward Euler scheme $\bar V_i$ and compare it to the output  of the DS algorithm $\mathcal{\widehat U}_i$. For this we introduce for $i=0,1,\dots,N-1$
\begin{align}
	&V_i := \mathbb{E}_i\Big[\mathcal{\widehat U}_{i+1}(X_{{i+1}})-  f\big(t_{i},X_{{i}},\mathcal{\widehat U}_{i+1}(X_{i+1}),
 \mathbb{E}_i[\sigma(X_{i})^\top D_x \mathcal{\widehat U}_{i+1}(X_{i+1})],\label{eq:Vi}  \mathbb{E}_i [\mathcal{I}[\mathcal{\widehat U}_{i+1},X_i](X^{c}_{i+1})]\big) \hspace{0.3mm} \Delta t   \Big] \qquad
\end{align}	
 and $V_N=\mathcal{\widehat U}_N(X_N)$. By the Markov property of $X=(X_{i})_{i =0,1,\dots,N}$ we have $V_i=v_i(X_i)$ for some functions $v_i: \R^d \to \R$, $i=0,1,\dots,N-1$, and we can introduce the $L^2$-approximation error  of the functions $v_i$ and $g$ in the class $\mathcal{N}^{\gamma,\eta}$ by
\begin{equation} \label{eq:eps-gamma-eta}
	\epsilon_i^{\gamma,\eta}=
	\begin{cases}
	\inf_{\mathcal U \in \mathcal{N}^{\gamma,\eta}} \mathbb{E} |v_i(X_i)-\mathcal{U}(X_i)|^2, \quad i=0,1,\dots,N-1  ,\\
		\inf_{\mathcal U \in \mathcal{N}^{\gamma,\eta}} \mathbb{E} |g(X_N)-\mathcal{U}(X_N)|^2 \, ,  \quad i=N  .
	\end{cases}
\end{equation}

The next theorem is the main result of our paper.

\begin{theorem}[Bound on the approximation error]\label{th:approx_result}
	Let the conditions (1)-(4) in Assumption~\ref{ass:conditions} hold and assume that $\mathcal{X}_0 \in L^4(\mathcal{F}_0,\R^d)$. Let $\widehat{\mathcal{ U}}_i \in \mathcal{N}^{\gamma,\eta}$, $0 \le i \le N$,  be the output of the DS scheme. Then, there exist constants $C>0$ (depending on $b, \sigma, \nu(\R^d), f,g,d,T,\mathcal{X}_0$) and $ C({\gamma,\eta})> 0 $ (depending on the Lipschitz constants $\gamma$, $\eta$)  such that in the limit $|\pi| \rightarrow 0$
	\begin{align} \label{eq:approx_error}
	\sup_{i \in \{0,1,\dots,N\}} \mathbb{E} \big|  Y_{t_i}-\widehat{\mathcal{ U}}_i(X_i) \big|^2 \le C\Big(|\pi| + \epsilon^Z(\pi)+ \epsilon^\Gamma(\pi) + C({\gamma,\eta}) |\pi|  + \epsilon_N^{\gamma,\eta} + N \sum_{i=0}^{N-1} \epsilon_i^{\gamma,\eta}\Big) \, .
	\end{align}
\end{theorem}

The first three terms in \eqref{eq:approx_error} correspond to the approximation error of the backward Euler scheme. Note that when $g \in \mathcal{N}^{\gamma,\eta}$ then one can initialize the scheme with $\mathcal{ \widehat U}_N=g$ and the term $\epsilon_N^{\gamma,\eta}$ vanishes.

Before proceeding to the proof of Theorem~\ref{th:approx_result} we show that in the linear case discussed in Remark~\ref{remark:linearcase} there is a simpler  bound on the approximation error.
\begin{proposition}
Let the conditions (1)-(4) in Assumption~\ref{ass:conditions} hold and assume that $\mathcal{X}_0 \in L^4(\mathcal{F}_0,\R^d)$. If $f=f(t,x)$ then
	\begin{align} \label{eq:approx_error_linear}
	\sup_{i \in \{0,1,\dots,N\}} \mathbb{E} \big|  Y_{t_i}-\widehat{\mathcal{ U}}_i(X_i) \big|^2 \le C\Big(|\pi| + \epsilon^Z(\pi)+ \epsilon^\Gamma(\pi) + \bar \epsilon_i \Big) \, ,
	\end{align}
	where $\bar \epsilon_i$ is the $L^2$-approximation error of $\bar V_i =  \mathbb{E}_i[ H_i] $ (see \eqref{eq:def-Hi}), which is defined by	$\bar \epsilon_i := \inf_{\mathcal U \in \mathcal{N}^{\gamma,\eta}} \mathbb{E} |\bar V_i-\mathcal{U}(X_i)|^2$, $i=0,1,\dots,N-1 \,$.
\end{proposition}

\begin{proof}
The approximation error is bounded from above by a sum of two terms,
\begin{align} \label{eq:linbound1}\mathbb{E} \big| Y_{t_i}-\widehat{\mathcal{ U}}_i(X_i) \big|^2 \le 2 \mathbb{E} \big|  Y_{t_i}- \bar V_i \big|^2 +
        2 \mathbb{E} \big| \bar V_i - \widehat{\mathcal{ U}}_i(X_i) \big|^2 \, .
\end{align}
The first term is the approximation error of the backward Euler scheme, see Proposition~\ref{prop:backward-Euler}. Now note that 
$ \widehat{\mathcal{ U}}_i(X_i) \in L^2 (\Omega, \F_{t_i}, P)$ and that $\bar V_i$ is the projection of  $H_i$ on $L^2(\Omega, \F_{t_i}, P)$. Hence 
\begin{equation} \label{eq:lin_bound}
\mathbb{E} \big| \bar H_i - \widehat{\mathcal{ U}}_i(X_i) \big|^2 = \mathbb{E} \big| H_i - \bar V_i \big |^2 +
    \mathbb{E} \big|\bar V_i -  \widehat{\mathcal{ U}}_i(X_i) \big|^2 \, .
\end{equation}
 It follows from \eqref{eq:lin_bound} that $\widehat{\mathcal{ U}}_i$, which is the minimizer of $ L_i^{\text{lin}} (\mathcal{U})=\mathbb{E} \big|H_i -  {\mathcal{ U}}(X_i) \big|^2 $ by definition, is also the  minimizer of the loss function $\mathcal{U} \mapsto  \mathbb{E} \big | \bar V_i - \mathcal{U}(X_i) \big |^2$, so that the second term in \eqref{eq:linbound1} equals $\bar \epsilon_i$.
\end{proof}

\begin{remark} A logical next step is to construct an approximating sequence with
$\sup_{i \in \{0,1,\dots,N\}} \mathbb{E} \big|\,Y_{t_i}-\widehat{\mathcal U}_i(X_i) \big|^2 = O(1/N)$. For this one has to  
assume that Assumption~\ref{ass:conditions}~(5) holds,  so that $\epsilon^Z(\pi)$ is of order $O(1/N)$. Moreover, and this is the difficult point, we need to control the approximation error $\bar \epsilon_i$ by a proper choice of the network functions. In the \emph{linear case} we need to ensure that the errors  $\bar \epsilon_i$ are of order $N^{\, {-}1}$.  The existence of such networks is ensured by universal approximation theorems such \citet{hornik1989multilayer}, \cite{hornik1990universal}. \citet{bib:gonon2021deep} give conditions which ensure that the approximation error can be made small without curse of dimensionality, i.e. by using networks with size growing only polynomial in the dimension $d$. 

In the general (nonlinear)  case  things are more involved: we need to construct networks (find sets  $\mathcal{N}^{\gamma,\eta}$) such that the errors $\epsilon_i^{\gamma,\eta}$ are of order $N^{\, {-} 3}$. Moreover, the functions $v_i$ depend on the choice of the network  $\mathcal{N}^{\gamma,\eta}$ so that control of the approximation error is a non-standard approximation problem. In \cite{germain2022approximation} this problem is used by considering Groupsort networks. However, these networks are not everywhere differentiable so that they cannot be employed to study the DS algorithm. Moreover, in the practical implementation of the method \cite{germain2022approximation} use deep (multilayer) networks for which no results on the approximation error is available. For these reasons we leave the analysis of the errors $\epsilon_i^{\gamma,\eta}$ for future research  and rely on numerical case studies to gauge accuracy and performance of our methodology (see Section \ref{sec:numerical}). 
\end{remark}

\section{Proof of Theorem~\protect{\ref{th:approx_result}}} \label{sec:proof}

Our proof  follows \cite{germain2022approximation} where the case without jumps is studied. However, we need to incorporate substantial  changes due to the non-local character of our problem. In the sequel we assume w.l.o.g that $\rho(z) \equiv 1$. Throughout this paper $C>0$ will denote a fixed constant, only depending on the dimension and parameters, but not on a partition. It may change from one line to another. To stress dependence on other parameters we will write $C(\cdot)$; for example to point out that the constant is dependent on the dimension $d$ we  write $C(d)$.

The starting point of the proof is the following decomposition of the approximation error $\mathbb{E}|Y_{t_i}-\mathcal{\widehat U}_i(X_{i})|^2$ into three terms
\begin{align} \label{eq:basic-error-decomp}
\mathbb{E}|Y_{t_i}-\mathcal{\widehat U}_i(X_{i})|^2 \le 3 \Big(\mathbb{E} |Y_{t_i}-\bar V_i|^2+\mathbb{E} |\bar V_{i}- V_i|^2+\mathbb{E} | V_{i}- \mathcal{\widehat U}_i(X_i)|^2  \Big) \, .
\end{align}
The first term is the approximation error from the explicit Euler scheme.  The following lemmas will be instrumental for deriving bounds on the other two terms.

\begin{lemma}  \label{lemma:tildeL}
	Define the loss function $\widetilde{L}_i(\mathcal{U}_i):= \mathbb{E} \big |\, V_i-\mathcal{U}_i(X_i)+\Delta f_i \Delta t \big|^2$, where we let
\begin{align*}
	\Delta f_i&:=f\big(t_{i},X_{{i}}, \mathcal{\widehat U}_{i+1}(X_{i+1}),\mathbb{E}_i [\sigma(X_{i})^\top D_x \mathcal{\widehat U}_{i+1}(X_{i+1})],\mathbb{E}_i [\mathcal{I}[\mathcal{\widehat U}_{i+1},X_i](X^{c}_{i+1})]\big)\\
&\qquad -f\big(t_{i},X_{{i+1}},\mathcal{\widehat U}_{i+1}(X_{i+1}),\sigma(X_{i})^\top D_x \mathcal{\widehat U}_{i+1}(X_{i+1}),\mathcal{I}[\mathcal{\widehat U}_{i+1},X_i](X^{c}_{i+1})\big) \, .
\end{align*}
	Then $\widehat {\mathcal{U}_i} \in \mathcal{N}^{\gamma,\eta}$ is a minimizer of $\widetilde{L}_i$ if and only if it is a minimizer of the loss function  $L_i$ used in the DS algorithm (see  \eqref{eq:loss_fun}).
\end{lemma}
\begin{proof}
	Fix $i \in \{0,1,\dots,N-1\}$. By  the martingale representation theorem (Lemma~\ref{lemma:martingalerep}) there exists $(\widehat Z,\widehat U) \in L_W^2(\R^d) \times L^2_J(\R)$ such that
	\begin{align}
	\mathcal{\widehat U}_{i+1}(X_{{i+1}})& -f\big(t_{i},X_{{i}},\mathcal{\widehat U}_{i+1}(X_{i+1}),\mathbb{E}_i[\sigma(X_{i})^\top D_x \mathcal{\widehat U}_{i+1}(X_{i+1})],\mathbb{E}_i \big [\mathcal{I}[\mathcal{\widehat U}_{i+1},X_i](X^{c}_{i+1})\big]\big) \Delta t \\
	&=V_i+\int_{t_i}^{t_{i+1}} \widehat Z_s \, \ud W_s + \int_{t_i}^{t_{i+1}} \widehat U_s(z) \, \widetilde{J}(\ud s, \ud z) \, .
	\end{align}
	Plugging this representation for $\mathcal{\widehat U}_{i+1}(X_{{i+1}})$ into the loss function $L_i(\mathcal{U}_i)$ from \eqref{eq:loss_fun} gives 	
	\begin{align}\nonumber
L_i(\mathcal{U}_i)&=\mathbb{E}\Big|V_i-\mathcal{U}_{i}(X_{i})+\Delta f_i \Delta t  +\int_{t_i}^{t_{i+1}} \widehat Z_s \, \ud W_s + \int_{t_i}^{t_{i+1}} \widehat U_s(z) \, \widetilde{J}(\ud s, \ud z)\,\Big|^2
\\	\label{eq:loss_fun_tilde} &=\widetilde{L}_i(\mathcal{U}_i)+\mathbb{E}\Big[\int_{t_i}^{t_{i+1}} | \widehat Z_s|_2^2 \, \ud s \Big] + \mathbb{E}\Big[\int_{t_i}^{t_{i+1}} \int_{\R^d} | \widehat U_s(z)|^2 \, \nu(\ud z)  \ud s \Big]
\\& \qquad  \qquad+ 2 \Delta t \mathbb{E} \Big[\Delta f_i  \Big(\int_{t_i}^{t_{i+1}} \widehat Z_s \, \ud W_s + \int_{t_i}^{t_{i+1}} \widehat U_s(z) \, \widetilde{J}(\ud s, \ud z)\Big)\Big] \, ,
	\end{align}
	where we used the It\^o Isometry (Lemma~\ref{lemma:isometry}) in the last step. As  the loss functions \eqref{eq:loss_fun_tilde} depends  on $\mathcal{U}_i$ only via $\widetilde{L}_i(\mathcal{U}_i)$,  $L_i(\mathcal{U}_i)$ and $\widetilde{L}_i(\mathcal{U}_i)$ are both minimized by the same   $\widehat {\mathcal{U}_i} \in \mathcal{N}^{\gamma, \eta}$.
\end{proof}

The next lemma gives an estimate for the last term in \eqref{eq:basic-error-decomp}.
\begin{lemma}  \label{lemma:upper_term3} There is a constant $C(\gamma,\eta)>0$ such that
	\begin{align}
	\mathbb{E} \big | V_i-\mathcal{\widehat U}_i(X_i)\big |^2 \le C \big(\epsilon_i^{\gamma,\eta} +  C({\gamma,\eta})|\Delta t_i|^3\big)\, .
	\end{align}
\end{lemma}
\begin{proof}
Recall that $\widetilde{L}_i(\mathcal{U}_i) = \mathbb{E} \big| V_i-\mathcal{U}_i(X_i)+ \Delta t \Delta f_i  \big|^2$  and  let $a = V_i - \mathcal{U}_i(X_i)$ and $b= \Delta t \Delta f_i$.  An application of the Young inequality in the form $(a+b)^2 \geq \frac{1}{2}a^2-b^2$ leads to
\begin{equation}\label{eq:ineq1}
	\widetilde{L}_i(\mathcal{U}_i)+ |\Delta t|^2 \mathbb{E}|\Delta f_i|^2 \geq \frac{1}{2}\mathbb{E} |V_i-\mathcal{U}_i(X_i)|^2 \, ,
\end{equation}

Next we derive  an upper estimate on $\widetilde{L}_i(\mathcal{U}_i)+ |\Delta t|^2 \mathbb{E}|\Delta f_i|^2$. Recall that for $a_1,\dots,a_n \in \R$ it holds that  $(a_1+\dots+a_n)^2 \le n(a_1^2 +\dots +a_n^2)$. This gives for $n=2$
\begin{align} \label{eq:lemma_eq1}	
	\widetilde{L}_i(\mathcal{U}_i)&+ |\Delta t|^2 \mathbb{E}|\Delta f_i|^2
\le 2 \mathbb{E} |V_i-\mathcal{U}_i(X_i)|^2 + 3 |\Delta t|^2 \mathbb{E}|\Delta f_i|^2 \, .
\end{align}
Moreover, we get, using the above inequality with $n=3$ and the Lipschitz property  of $f$,
\begin{align*}
 \mathbb{E}|\Delta f_i|^2  &\le    3  {[f]}_L  \Big( \mathbb{E}|X_{i+1}-X_i|_2^2
	 + \mathbb{E}\big | \sigma(X_i)^\top D_x\mathcal{ \widehat{U}}_{i+1}(X_{i+1})-\mathbb E_i[\sigma(X_i)^\top D_x\mathcal{\widehat{U}}_{i+1}(X_{i+1})]\big |_2^2
\\ & \qquad \qquad   +\mathbb{E}\big|\mathcal{I}[\mathcal{\widehat U}_{i+1},X_i](X^{c}_{i+1})-\mathbb E_i[\mathcal{I}[\mathcal{\widehat U}_{i+1},X_i](X^{c}_{i+1})]\big |^2 	\Big) \\
	& \le 3 {[f]}_L \Big( \mathbb{E}|X_{i+1}-X_i|_2^2   + \mathbb{E}\big[|\sigma(X_i)^\top|_2^2 \big|D_x\mathcal{ \widehat{U}}_{i+1}(X_{i+1})-D_x\mathcal{ \widehat{U}}_{i+1}(X_{i})\big|_2^2\big]\\
& \qquad \qquad  +\bE \big|\mathcal{I}[\mathcal{\widehat U}_{i+1},X_i](X^{c}_{i+1})-\mathcal{I}[\mathcal{\widehat U}_{i+1},X_i](X_{{i}})\big |^2
	\Big)  \, ,
	\end{align*}
where we use in the last inequality the $L^2$-minimization property of conditional expectation $\bE_i[\cdot]$ and the  inequality $ |\sigma(x)^\top z|_2^2 \le  |\sigma(x)^\top|^2_2 |z|^2_2$.
Using the standard estimates of the Euler-Maryuama scheme  gives
\begin{equation}\label{eq:lemma_eq2}
\mathbb{E}|X_{i+1}-X_i|_2^2 \le C(1+\|\mathcal{X}_0\|_2^2)\Delta t\,.
\end{equation}
 Recall that $\mathcal{\widehat U}_{i+1}$ belongs to  $ \mathcal{ N}^{\gamma,\eta} $ so that
 $|D_x \mathcal{\widehat{U}}_{i+1}(X_{i+1})-D_x\mathcal{ \widehat{U}}_{i+1}(X_{i})|_2 \le \eta |X_{i+1}- X_i|_2$. Moreover,
 $|\sigma(x)|^2_2 \le C(1+|x|^2_2)$.  Hence we get from the standard estimates of the Euler-Maryuama scheme that
\begin{align}\label{eq:lemma_eq3}
	 \mathbb{E}\big[|\sigma(X_i)|_2^2\mathbb E_i\big|D_x\mathcal{ \widehat{U}}_{i+1}(X_{i+1})-D_x\mathcal{ \widehat{U}}_{i+1}(X_{i})\big|^2\big] \le C(\eta)(1+\|\mathcal{X}_0\|_4^2)^2\Delta t\,.
\end{align}
Finally we have, using that  $|\mathcal{\widehat U}_{i+1}(x)-\mathcal{\widehat U}_{i+1}(x')|\le C(\gamma)(1+|x|_2+|x'|_2)$ and  $|\gamma^X(x,z)|_2 \le C(1+|x|_2)$,	
\begin{align*}
	\big | \mathcal{I}[\mathcal{\widehat U}_{i+1},& \tilde  x](x)-\mathcal{I}[\mathcal{\widehat U}_{i+1},\tilde x](x')  |
\\&\le \int_{\R^d} \big[ \, |\mathcal{\widehat U}_{i+1}(x+\gamma^X(\tilde{x},z))-\mathcal{\widehat U}_{i+1}(x'+\gamma^X(\tilde{x},z)) | + |\mathcal{\widehat U}_{i+1}(x)-\mathcal{\widehat U}_{i+1}(x')|  \,  \big] \,  \nu (\ud z)  \\
	&\le C(\gamma)  \int_{\R^d} [1+|\gamma^X(\tilde{x},z)|_2+|x|_2+|x'|_2] \, \, \nu(\ud z) \, |x-x'|_2	\\
	& \le C(\gamma) \nu(\R^d) (1+|\tilde{x}|_2+|x|_2+|x'|_2)|x-x'|_2	\, .
\end{align*}
	By the Cauchy-Schwarz inequality and the standard estimates of the Euler-Maruyama scheme we get
\begin{align}	\label{eq:lemma_eq4}
	\bE|\mathcal{I}[\mathcal{\widehat U}_{i+1},X_i](X^{c}_{i+1})-\mathcal{I}[\mathcal{\widehat U}_{i+1},X_i](X_{i})|^2
 \le C(\gamma) (1+\|\mathcal{ X}_0\|_4^2)^2 \Delta t\, .
\end{align}
Plugging  \eqref{eq:lemma_eq2}, \eqref{eq:lemma_eq3}, \eqref{eq:lemma_eq4} into \eqref{eq:lemma_eq1} we get
	\begin{align} \label{eq:ineq2}
	\widetilde{L}_i(\mathcal{U}_i)+| \Delta t|^2 \mathbb{E}|\Delta f_i|^2 \le 2 \Big( \mathbb{E} |V_i-\mathcal{U}_i(X_i)|^2 + C({\gamma,\eta})|\Delta t|^3\Big) \, .
	\end{align}
	By applying inequality \eqref{eq:ineq1} to $\mathcal{U}_i=\mathcal{\widehat U}_i$ we get for generic $\mathcal{U}_i$
	\begin{align} \label{eq:ineq21}
	 \frac{1}{2}\mathbb{E} |V_i-\mathcal{\widehat U}_i(X_i)|^2 \le \widetilde{L}_i(\mathcal{\widehat U}_i)+| \Delta t|^2 \mathbb{E}|\Delta f_i|^2  \le  \widetilde{L}_i(\mathcal{ U}_i)+| \Delta t|^2 \mathbb{E}|\Delta f_i|^2
	\end{align}	
where the second inequality follows as $\widehat{\mathcal{U}}_i$ is a minimizer of ${L}_i$ and hence by Lemma~\ref{lemma:tildeL} a minimizer of $\widetilde{L}_i$.
	Combining this with \eqref{eq:ineq2} gives
\begin{align} \label{eq:ineq3}
	\mathbb{E} |V_i-\mathcal{\widehat U}_i(X_i)|^2 \le C \Big( \mathbb{E} |V_i-\mathcal{U}_i(X_i)|^2 +  C({\gamma,\eta})|\Delta t|^3\Big) \,. 	\end{align}
	By minimizing over all $\mathcal U_i \in \mathcal{ N}_i^{\gamma,\eta}$ we get $
	\mathbb{E} |V_i-\mathcal{\widehat U}_i(X_i)|^2 \le C \Big(\epsilon_i^{\gamma,\eta} +  C({\gamma,\eta})|\Delta t|^3\Big)$.

\end{proof}
The next two lemmas are needed to estimate the middle term in \eqref{eq:basic-error-decomp}, i.e.~$\mathbb{E}\big | \bar V_i - V_i \big |^2$.

\begin{lemma} \label{lemma:upperZ} It holds that
	\begin{align}\label{eq:ineqZ}
	 \mathbb{E}\big |\mathbb{E}_i[\sigma(X_{i}) D_x \mathcal{\widehat U}_{i+1}(X_{i+1})]-\bar Z_i \big |^2_2 &
   \le
\frac{d}{\Delta t} \mathbb{E} \Big ( \var_i \big ( \mathcal{\widehat U}_{i+1}(X_{i+1})-\bar V_{i+1}\big ) \Big) \, .
	\end{align}
\end{lemma}
\begin{proof}
The proof of the lemma can be reduced to the proof of the corresponding result in the no-jump case by conditioning on the jump component
$X_{i+1}^J$ in the decomposition  $X_{i+1} = X_{i+1}^c + X_{i+1}^J$. 	 Using a standard integration by parts argument (see e.g. \citet{fahim2011probabilistic}) we get
	\begin{align*}
	 \mathbb{E}_i \big[ \sigma(X_i)^\top D_x \mathcal{\widehat U}_{i+1}(X_{i+1}) \big | X_{i+1}^J \big]
&=\mathbb{E}_i \Big[ \mathcal{\widehat U}_{i+1}(X_{i+1})\frac{\Delta W_{i}}{\Delta t} \big | X_{i+1}^J\Big]
	\end{align*}
and hence $\mathbb{E}_i  \big[ \sigma(X_i)^\top D_x \mathcal{\widehat U}_{i+1}(X_{i+1}) \big] = \mathbb{E}_i \Big[ \mathcal{\widehat U}_{i+1}(X_{i+1})\frac{\Delta W_{i}}{\Delta t}  \Big]$.
Recall now the definition of  $\bar Z_i$ in the backward Euler scheme. We get
	\begin{align}
	&\mathbb{E}_i \big[ \sigma(X_i)^\top D_x \mathcal{\widehat U}_{i+1}(X_{i+1}) \big]-\bar Z_i =\mathbb{E}_i \Big[ \big( \mathcal{\widehat U}_{i+1}(X_{i+1})-\bar V_{i+1}\big) \frac{\Delta W_{i}}{\Delta t} \Big]\\
	&\qquad =\mathbb{E}_i \Big[ \big( \mathcal{\widehat U}_{i+1}(X_{i+1})-\bar V_{i+1}-\mathbb{E}_i[\mathcal{\widehat U}_{i+1}(X_{i+1})-\bar V_{i+1}]\big) \frac{\Delta W_{i}}{\Delta t} \Big] \, .
	\end{align}
	By the Cauchy-Schwarz inequality we obtain
	\begin{align} \label{eq:step2_1}
	\mathbb{E} \Big[ \mathbb{E}_i \Big[ &\big( \mathcal{\widehat U}_{i+1}(X_{i+1})-\bar V_{i+1}-\mathbb{E}_i[\mathcal{\widehat U}_{i+1}(X_{i+1})-\bar V_{i+1}]\big) \frac{\Delta W_{i}}{\Delta t} \Big]^2\Big]\\&\leq \frac{1}{|\Delta t|^2} \bE \Big[\mathbb{E}_i \big|\Delta W_{i}\big|^2 \mathbb{E}_i \big| \mathcal{\widehat U}_{i+1}(X_{i+1})-\bar V_{i+1}-\mathbb{E}_i[\mathcal{\widehat U}_{i+1}(X_{i+1})-\bar V_{i+1}]\big | ^2 \Big]
	\\&=\frac{d}{\Delta t}  \mathbb{E}\Big ( \var_i  \big( \mathcal{\widehat U}_{i+1}(X_{i+1})-\bar V_{i+1}\big ) \Big)
    \end{align}
	where we use $ \mathbb{E}_i |\Delta W_i |^2 = d \Delta t$ and the definition of the conditional variance in the last step.
\end{proof}

\begin{lemma}~\label{lemma:upperGamma} It holds that
	\begin{align} \label{eq:upperGamma}
\mathbb{E}\big | \mathbb{E}_i [\mathcal{I}[\mathcal{\widehat U}_{i+1},X_i](X^{c}_{i+1})]-\bar \Gamma_i\big |^2
&\le \frac{\nu(\R^d)}{\Delta t}  \mathbb{E} \Big ( \var_i \big ( \mathcal{\widehat U}_{i+1}(X_{i+1})-\bar V_{i+1}\big ) \Big)+O(|\Delta t_i|^2) \, .
\end{align}
\end{lemma}

\begin{proof}
 The l.h.s. of \eqref{eq:upperGamma} is bounded by
	\begin{align}
	2\underbrace{\bE \Big|\mathbb{E}_i \Big[\mathcal{I}[\mathcal{\widehat U}_{i+1},X_i](X^{c}_{i+1})- \mathcal{\widehat U}_{i+1}(X_{{i+1}}) \frac{\Delta M_{i}}{\Delta t} \Big] \Big |^2}_{\text{(A)}}
	 +2\underbrace{\mathbb{E}  \Big | \mathbb{E}_i \Big[  \mathcal{\widehat U}_{i+1}(X_{{i+1}}) \frac{\Delta M_{i}}{\Delta t} \Big]-\bar \Gamma_i \Big |^2}_{\text{(B)}} \, .
\end{align}
\textit{Upper bound for term (A).} Recall that $X_{i+1} = X_{i+1}^c + X_{i+1}^J$,  where the jump term equals
$$X_{i+1}^J= \int_{t}^{t_{i+1}} \int_{\R^d} \gamma^X(X_{i},z) \, \widetilde{J}(\ud s,\ud z).$$
Denote by $\Delta N_{i} = J((t_i,t_{i+1}]\times \R^d)$ the number of jumps  in the interval $(t_i,t_{i+1}]$ and note that $\Delta N_{i}$ is Poisson distributed with parameter $\bar \nu :=\nu(\R^d)\Delta t$ such that $\Delta M_{i}=\Delta N_{i}-\bar \nu$.  Let $(\xi_j)_{j \in \bN}$ be a sequence of iid random variables with distribution $\frac{1}{\nu(\R^d)}\nu(\ud z)$.	
By conditioning on  $\Delta N_{i}$  we get
\begin{align}
	\mathbb{E}_i \Big[ \mathcal{\widehat U}_{i+1}(X_{{i+1}}) \frac{\Delta M_{i}}{\Delta t} \Big] 
  =& e^{-\bar \nu} \sum_{k=0}^\infty \bE_i \Big[ \mathcal{\widehat U}_{i+1}\big(X^c_{{i+1}}+\sum_{j=1}^{k}\gamma^X(X_i,\xi_j)\Big) \frac{(k-\bar \nu)}{\Delta t}\Big ]  \frac{\bar \nu^k}{k!} \\
	= & e^{-\bar \nu} \Big\{ \bE_i \big[  \mathcal{\widehat U}_{i+1}(X_{{i+1}}^c) (-\nu(\R^d)) \big]+
        \bE_i \Big[  \mathcal{\widehat U}_{i+1}\big(X_{{i+1}}^c+\gamma^X(X_i,\xi_1)\big) (1-\bar \nu) \Big]\nu(\R^d)\\
     & + \sum_{k=2}^\infty \bE_i \Big[   \mathcal{\widehat U}_{i+1}\Big(X^c_{{i+1}}+\sum_{j=1}^{k}\gamma^X(X_i,\xi_j)\Big) \frac{(k-\bar \nu)}{\Delta t}  \Big] \frac{\bar \nu^k}{k!}\Big\}\\
	= &e^{-\bar \nu} \bE_i \Big[ \int_{\R^d}   [\mathcal{\widehat U}_{i+1}(X_{{i+1}}^c+\gamma^X(X_i,z)) -  \mathcal{\widehat U}_{i+1}(X_{{i+1}}^c)]   \nu(\ud z) \Big]+O(\Delta t) \, , \label{eq:form1}
\end{align}
	where used in the last equality that
	$\bE_i \big[  \mathcal{\widehat U}_{i+1}(X_{{i+1}}^c+\gamma^X(X_i,\xi_1)) \big]=\bE_i \big[ \int_{\R^d}  \mathcal{\widehat U}_{i+1}(X_{{i+1}}^c+\gamma^X(X_i,z)) \frac{\nu(\ud z)}{\nu(\R^d)} \, \big ]. $	
	Using  \eqref{eq:form1} we get that
		\begin{align} \label{eq:errorOt}
	\bE \Big  |& \mathbb{E}_i \Big[\mathcal{I}[\mathcal{\widehat U}_{i+1},X_i](X^{c}_{i+1})- \mathcal{\widehat U}_{i+1}(X_{{i+1}}) \frac{\Delta M_{i}}{\Delta t} \Big]    \Big |^2
\\&
=  \bE \Big |\mathbb{E}_i \Big[\int_{\R^d} (1-e^{-\bar \nu} )  \big[\mathcal{\widehat U}_{i+1}\big(X^c_{i+1}+\gamma^X(X_{{i}},z)\big)-\mathcal{\widehat U}_{i+1}(X^c_{i+1}) \big] \, \nu( \ud z)    \Big]  +O(\Delta t) \Big |^2
\\&
\le 2 (1-e^{-\bar \nu} )^2 \bE \Big |\mathbb{E}_i \Big[\int_{\R^d}   \big[\mathcal{\widehat U}_{i+1}\big(X^c_{i+1}+\gamma^X(X_{{i}},z)\big)-\mathcal{\widehat U}_{i+1}(X^c_{i+1}) \big] \, \nu( \ud z)    \Big] \Big |^2 + O(|\Delta t|^2)\,.
\end{align}
Note now that $(1-e^{-\bar \nu})^2 = (\nu(\R^d)\Delta t)^2 +O(|\Delta t|^3)$ is of  order $|\Delta t|^2$.  Moreover, using  the Lipschitz property of $\mathcal{\widehat U}_{i+1}$, the estimate $|\gamma^X(x,z)|_2 \le  C(|x|_2+1)$, Jensen's inequality and the standard estimates for the Euler Maruyama scheme 	it is easily seen that
\begin{align*}
  \bE \Big |\mathbb{E}_i \Big[\int_{\R^d}   \big[\mathcal{\widehat U}_{i+1}\big(X^c_{i+1}+\gamma^X(X_{{i}},z)\big)-\mathcal{\widehat U}_{i+1}(X^c_{i+1}) \big]  \,  \nu( \ud z)    \Big] \Big |^2 < \infty
\end{align*}
is bounded. Hence we get that term (A) is of order $|\Delta t|^2$.\\
	
\noindent \textit{Upper bound for term (B).} Using the definition of $\bar \Gamma_i$ in the backward Euler scheme we get
	\begin{align}
	\mathbb{E}_i \Big[  \mathcal{\widehat U}_{i+1}(X_{{i+1}}) \frac{\Delta M_{i}}{\Delta t} \Big]-\bar \Gamma_i &=\mathbb{E}_i \Big[ \big( \mathcal{\widehat U}_{i+1}(X_{i+1})-\bar V_{i+1}\big) \frac{\Delta M_{i}}{\Delta t} \Big]\\
	& =\mathbb{E}_i \Big[ \big( \mathcal{\widehat U}_{i+1}(X_{i+1})-\bar V_{i+1}-\mathbb{E}_i[\mathcal{\widehat U}_{i+1}(X_{i+1})-\bar V_{i+1}]\big) \frac{\Delta M_{i}}{\Delta t} \Big].
	\end{align}
	By the Cauchy-Schwarz inequality, the It\^o isometry and the law of iterated expectations we have
\begin{align}\label{eq:step2_2}
	\mathbb{E} \Big[ \mathbb{E}_i \Big[ &\big( \mathcal{\widehat U}_{i+1}(X_{i+1})-\bar V_{i+1}-
        \mathbb{E}_i[\mathcal{\widehat U}_{i+1}(X_{i+1})-\bar V_{i+1}]\big) \frac{\Delta M_{i}}{\Delta t} \Big]^2\Big] \\
    &\leq \frac{1}{|\Delta t|^2}\mathbb{E} \big| \mathbb{E}_i \big|\Delta M_{i}\big|^2
        \mathbb{E}_i \big| \mathcal{\widehat U}_{i+1}(X_{i+1})-\bar V_{i+1}-\mathbb{E}_i[\mathcal{\widehat U}_{i+1}(X_{i+1})-\bar V_{i+1}]\big | ^2 \big|
	\\&=\frac{\nu(\R^d)}{\Delta t}  \mathbb{E}\big[ \var_i (\widehat {\mathcal{ U}}_{i+1}(X_{i+1})-\bar V_{i+1} ) \big ] \, ,
\end{align}
	where we used $ \mathbb{E}_i |\Delta M_{i}|^2=\frac{\nu(\R^d)}{\Delta t}$. Combining this with the estimate for term (A)  gives the result.
\end{proof}

\begin{proof}[Proof of Theorem~\ref{th:approx_result}]
	To find a bound of $	\sup_{i \in \{0,1,\dots,N\}} \mathbb{E}_i \big|  Y_{t}-\widehat{\mathcal{ U}}_i(X_i) \big|^2$ we decompose the approximation error into three terms
	\begin{align} \label{eq:error}
	\mathbb{E}|Y_{t}-\mathcal{\widehat U}_i(X_{i})|^2 \le 3 \big(\mathbb{E} |Y_{t}-\bar V_i|^2+\mathbb{E} |\bar V_{i}- V_i|^2+\mathbb{E} | V_{i}- \mathcal{\widehat U}_i(X_i)|^2  \big) \, .
	\end{align}
	The first term is the classical time discretization error \eqref{eq:time_dis_error} of the backward Euler scheme,  and an estimate  for the third term is given in Lemma~\ref{lemma:upper_term3}.

In the following we concentrate on  the second term. From the expressions of $\bar{V}_i$ and $V_i$ in \eqref{eq:expl_euler_scheme1} and in \eqref{eq:Vi} and by applying the Young inequality $(a+b)^2 \le (1+\beta)a^2+(1+\frac{1}{\beta})b^2$  with $\beta \in (0,1)$ we get
	\begin{align}\label{eq:ineq5}
	&\mathbb{E}\big|\bar{V}_i-V_i\big|^2=\mathbb{E}\big|\mathbb{E}_i \big[ \bar V_{i+1} - \mathcal{\widehat U}_{i+1}(X_{{i+1}}) \big] 	
	-\Delta t\mathbb{E}_i\big[ f(t,X_{i},\bar{V}_{i+1},\bar Z_i, \bar \Gamma_i) \\ & \qquad - f\big(t_{i},X_{{i}},\mathcal{\widehat U}_{i+1}(X_{i+1}),\mathbb{E}_i[\sigma(X_{i}) D_x \mathcal{\widehat U}_{i+1}(X_{i+1})],\mathbb{E}_i [\mathcal{I}[\mathcal{\widehat U}_{i+1},X_i](X^{c}_{i+1})]\big) \big ]\big|^2 \\
	&\leq (1+\beta)\mathbb{E}\big|\mathbb{E}_i \big[ \bar V_{i+1} - \mathcal{\widehat U}_{i+1}(X_{{i+1}}) \big]\big|^2
	+(1+{1}/{\beta})|\Delta t|^2\mathbb{E}\big| f(t,X_{i},\bar{V}_{i+1},\bar Z_i, \bar \Gamma_i) \\ & \qquad - f\big(t_{i},X_{{i}},\mathcal{\widehat U}_{i+1}(X_{i+1}),\mathbb{E}_i[\sigma(X_{i}) D_x \mathcal{\widehat U}_{i+1}(X_{i+1})],\mathbb{E}_i \mathcal{I}[\mathcal{\widehat U}_{i+1},X_i](X^{c}_{i+1})]\big) \big|^2  \\
	& 	 \leq (1+\beta)\mathbb{E}\big|\mathbb{E}_i \big[ \bar V_{i+1} - \mathcal{\widehat U}_{i+1}(X_{{i+1}}) \big]\big|^2
	+3[f]_L(1+{1}/{\beta})|\Delta t|^2\big( \mathbb{E}\big |\mathcal{\widehat U}_{i+1}(X_{i+1})-\bar{V}_{i+1} \big |^2 \\& \qquad
	+\mathbb{E}\big |\mathbb{E}_i[\sigma(X_{i}) D_x \mathcal{\widehat U}_{i+1}(X_{i+1})]-\bar Z_i \big |^2_2 +\mathbb{E}\big |\mathbb{E}_i [\mathcal{I}[\mathcal{\widehat U}_{i+1},X_i](X^{c}_{i+1})]-\bar \Gamma_i\big |^2 \big) \, ,
	\end{align}	
	where we used the Lipschitz condition, $(a+b+c)^2 \le 3(a^2+b^2+c^2)$ and the Cauchy-Schwarz inequality in the last step.
	With Lemma~\ref{lemma:upperZ} and Lemma~\ref{lemma:upperGamma} we obtain
	\begin{align}
	\mathbb{E}|\bar{V}_i-V_i|^2	 &\leq (1+\beta)\mathbb{E}\big|\mathbb{E}_i \big[ \bar V_{i+1} - \mathcal{\widehat U}_{i+1}(X_{{i+1}}) \big]\big|^2 + 3[f]_L(1+\beta)\frac{|\Delta t|^2}{\beta}\Big [ \mathbb{E}\big |\mathcal{\widehat U}_{i+1}(X_{i+1})-\bar{V}_{i+1} \big |^2 \nonumber
		\\& \quad +   C   |\Delta t|^2
	+\frac{\nu(\R^d)+d}{\Delta t}   \big( \mathbb{E}\big |\mathcal{\widehat U}_{i+1}(X_{i+1})-\bar V_{i+1}\big |^2-\mathbb{E}\big|\mathbb{E}_i[\mathcal{\widehat U}_{i+1}(X_{i+1})-\bar V_{i+1}]\big|^2\big) \Big ] \nonumber
	\\& \le (1+C\Delta t) \mathbb{E}\big |\mathcal{\widehat U}_{i+1}(X_{i+1})-\bar V_{i+1}\big |^2+C |\Delta t|^3 \, , \label{eq:ineq6}
	\end{align}
	by choosing explicitly $\beta=3[f]_L (\nu(\R^d)+d)\Delta t $ for $\Delta t$ small enough. By using again Young inequality on the r.h.s. of \eqref{eq:ineq6} with $\beta=\Delta t$
	\begin{align*}
		 \mathbb{E}\big |\mathcal{\widehat U}_{i+1}(X_{i+1})-\bar V_{i+1}\big |^2 \le (1+\Delta t)\mathbb{E}\big |\bar V_{i+1}-V_{i+1}\big |^2 + (1+1/\Delta t) \mathbb{E}\big |\mathcal{\widehat U}_{i+1}(X_{i+1})-V_{i+1}\big |^2
	\end{align*}
	 and since $\Delta t = O(1/N)$, we then get
	\begin{align*}
	\mathbb{E}|\bar{V}_i-V_i|^2	 &\leq  (1+C\Delta t) \mathbb{E}\big |\bar V_{i+1}-V_{i+1}\big |^2+ CN \mathbb{E}\big |\mathcal{\widehat U}_{i+1}(X_{i+1})-V_{i+1}\big |^2+C |\Delta t|^3 \, .
	\end{align*}
	As it holds with Lemma~\ref{lemma:tildeL} that \begin{align} \label{eq:term3}
	\mathbb{E} |V_i-\mathcal{\widehat U}_i(X_i)|^2 \le C \big(\epsilon_i^{\gamma,\eta} +  C({\gamma,\eta})|\Delta t|^3\big) \, ,
	\end{align}
	and  $\bar V_N=g(X_N)$ and $V_N=\mathcal{ \widehat U}_N(X_N)$, we deduce with the discrete Gronwall lemma that
	\begin{align} \label{eq:term2}
	\sup_{i \in \{0,1,\dots,N\}} \mathbb{E}|\bar{V}_i-V_i|^2 \le C \epsilon_N^{\gamma, \eta} &+CN \sum_{i=1}^{N-1}  \big(\epsilon_i^{\gamma,\eta} +  C({\gamma,\eta})|\Delta t|^3\big)+ C \Delta t ^2 \, .
	\end{align}
	The required bound for the approximation error on $Y$ follows by plugging \eqref{eq:time_dis_error}, \eqref{eq:term3},  and \eqref{eq:term2} in \eqref{eq:error}.
\end{proof}

\section{Numerical study} \label{sec:numerical}

We test our algorithm on two examples, a linear and a semilinear PIDE,  for  varying  dimension $d$. In each example we choose maturity $T=1$. The shallow networks from the previous section were chosen for mathematical convenience; for numerical reasons we prefer to work with a network with 2 hidden layers with $d+10$ neurons each.  We use batch normalization and and optimize the loss function using Adam gradient descent with exponentially decreasing learning rate. All computations were run on a Lenovo Thinkpad notebook with an Intel Core i5 processor (1.7 GHz) and 16 GB memory.

\subsection{Linear PIDE}

First we consider an example from \citet{bib:xu2009approximate} who study the pricing of  basket options in jump diffusion models. Basket  options are typically difficult to price due to the lack of an analytic solution. Monte Carlo simulation is a simple and accurate alternative, however, it is very time-consuming.

In this example we assume that underlying asset prices follow jump-diffusion processes with correlated Brownian motions and two types of Poisson jumps in the jump component: a systematic jump that affects all asset prices and idiosyncratic jumps that only affect specific asset. For simplicity we assume deterministic jump sizes. We consider Brownian motions $W^1,\dots,W^d$ with pairwise correlation $\bar \rho$ described by the $d\times d$ correlation matrix $\rho=(\rho_{ij})_{i,j=1,2,\dots,d}$ with $\rho_{ii}=1$ and $\rho_{ij}=\bar \rho$ for $i \neq j$. We assume independent Poisson processes $N^0,\dots,N^d$ with intensities $\lambda^0,\dots,\lambda^d$ and that the Brownian motions and Poisson processes are independent of each other. We consider a equally weighted portfolio composed of $d$ assets with asset prices $S^1,\dots,S^d$ satisfying
\begin{align*}
	\frac{\ud S^i_t}{S^i_t}=r \, \ud t + \sigma_i  \, \ud W^i_t+h^0_i \, \ud [N^0_t-\lambda^0 t] +h^1_i \, \ud [N^i_t-\lambda_i t] \, ,
\end{align*}
for $i=1,2,\dots,d$. Here $r$ is the risk-free interest rate and $\sigma_i$ is the volatility of assets $i$, and $h_i^0$, $h^1_i$ are percentage jump sizes of asset $i$ at jump times of Poisson processes $N^0$ and $N^i$, respectively. All
coefficients are assumed to be constant. The solution to the SDE
\begin{align*}
S_t^i=S_0^ie^{(r-\frac{1}{2}\sigma^2 -h^0_i \lambda^0-h^1_i \lambda_i)t+\sigma_i W^i_t+ \ln(1+h^0_i) N^0_t + \ln(1+h^1_i)N^1_t}.
\end{align*}
can be simulated directly.
The basket value at time $t$ is given by
$B_t =\frac{1}{d}\sum_{i=1}^{d} S^i_t$
and the basket call option price at time 0 is given by $u(0,x)=e^{-rT}\bE_{0}[\max (B_T-K,0)|S_0=x]$. The problem can be described with the PIDE
\begin{align*}
u_t(t,x) &+\sum_{i=1}^{d}(r-h_i^0\lambda^0-h_i^1 \lambda_i^1)x_i   u_{x_i} (t,x) + \lambda^0 [u(t,x  (1+h_0)  )-u(t,x)] \\ &    + \sum_{i=1}^{d}  \lambda_i^1 [u(t,x  (1+ \bar h^i  ))-u(t,x)] +\frac{1}{2}\sum_{i,j=1}^{d}  \sigma_i \sigma_j \rho_{ij} x_i x_j  u_{x_i x_j}(t,x) = 0    & \text{on } [0,T) \times \R^d \, , \\
&u(T,x)= g(x)  \qquad  \text{on } \R^d  \, ,
\end{align*}
where $g(x)=e^{-rT}(x-K)^+$ and $\bar h^i_j=h_i^1$ for $i=j$ and $\bar h^i_j=0$ else.

In our numerical experiments we work with the following parameter values
\begin{align*}
\begin{cases}
	r=0.05, \, \rho=0.2,    \, \sigma= (0.1,\dots,0.1)^T, \, K=1.2 \, ,\\
	h^0=h^1=(0.1,\dots,0.1)^T,  \,  \lambda^0=10, \, \lambda^1=(10,\dots,10)^T \, ,
\end{cases}
\end{align*}
and we compute the approximate solution $\widehat{ \mathcal{U}}_0$ by minimizing the loss function $L_i^{\text{lin}}$  from \eqref{eq:loss-function-linear} for $x \in [0,2]^d$. We use batch size 6000 and 10000 gradient descent iterations with a learning rate 0.01 that decreases after 2000, 4000, and 7000 iterations by factor $10^{{-}1}$. We apply the \texttt{softplus} activation function in the hidden layers and a linear activation function in the output layer.

 We provide two test cases for dimension $d=4$ and $d=10$. The estimate for $u(0,s)$ for $s=(1,\dots,1,x_d), \ x_d \in [0,2]$ and the loss functions of the training procedure are illustrated in Figure  \ref{fig:basketcalloption}. Note that there is more noise in the loss function for $d=4$ than for $d=10$ as we compute the mean over $d$ assets inside the payoff function $g$. Table \ref{tab:basketcalloption} reports the average estimate of $u(0,s_0)$ for $s_0=(1,\dots,1)^T$, and standard deviation observed over 10 independent runs. For $d=4$ and $d=10$ one run takes approximately 150 seconds resp. 550 seconds.  The reference solutions are computed with Monte-Carlo using $10^6$ simulations and are marked as grey dots in Figure \ref{fig:basketcalloption}. Computation of the reference solution via Monte Carlo took around 1.2 resp. 2 seconds per point. This example clearly shows the advantages of the DNN method over standard Monte Carlo for computing the solution on the entire set  $A$. Suppose that we want to compute the solution on $[0,2]^4$ (as in this example). Even the very coarse grid $\{0,0.5,1,1.5,2\}^4$ has already 625 gridpoints, and  computing the solution for each gridpoint takes approximately $1.2 \times 625 = 750$ seconds, which is already about five times the time for training the network.

\begin{table}[t]
	\begin{center}
		\begin{tabular}{l c c c c  }
			& Averaged value & Standard deviation & MC solution & Relative error (\%) \\
			\hline
			\textbf{$d=4$} 	&	0.093935	& 0.000227 &  0.09150 & 2.59 \\
			\textbf{$d=10$} &	0.083854 &  0.000107   & 	0.08236 & 1.81 \\
						\hline
		\end{tabular}
	\end{center}
	\vspace{2mm}
	\caption {Average estimates, standard deviations over 10 independent runs and Monte Carlo solutions using $10^6$ simulations used as benchmark for the relative error are reported.}   \label{tab:basketcalloption}
\end{table}

\begin{figure} %
	\centering
	\subfloat[Estimate of $u(0,s)$ for $d=4$. ]{\includegraphics[scale=0.32]{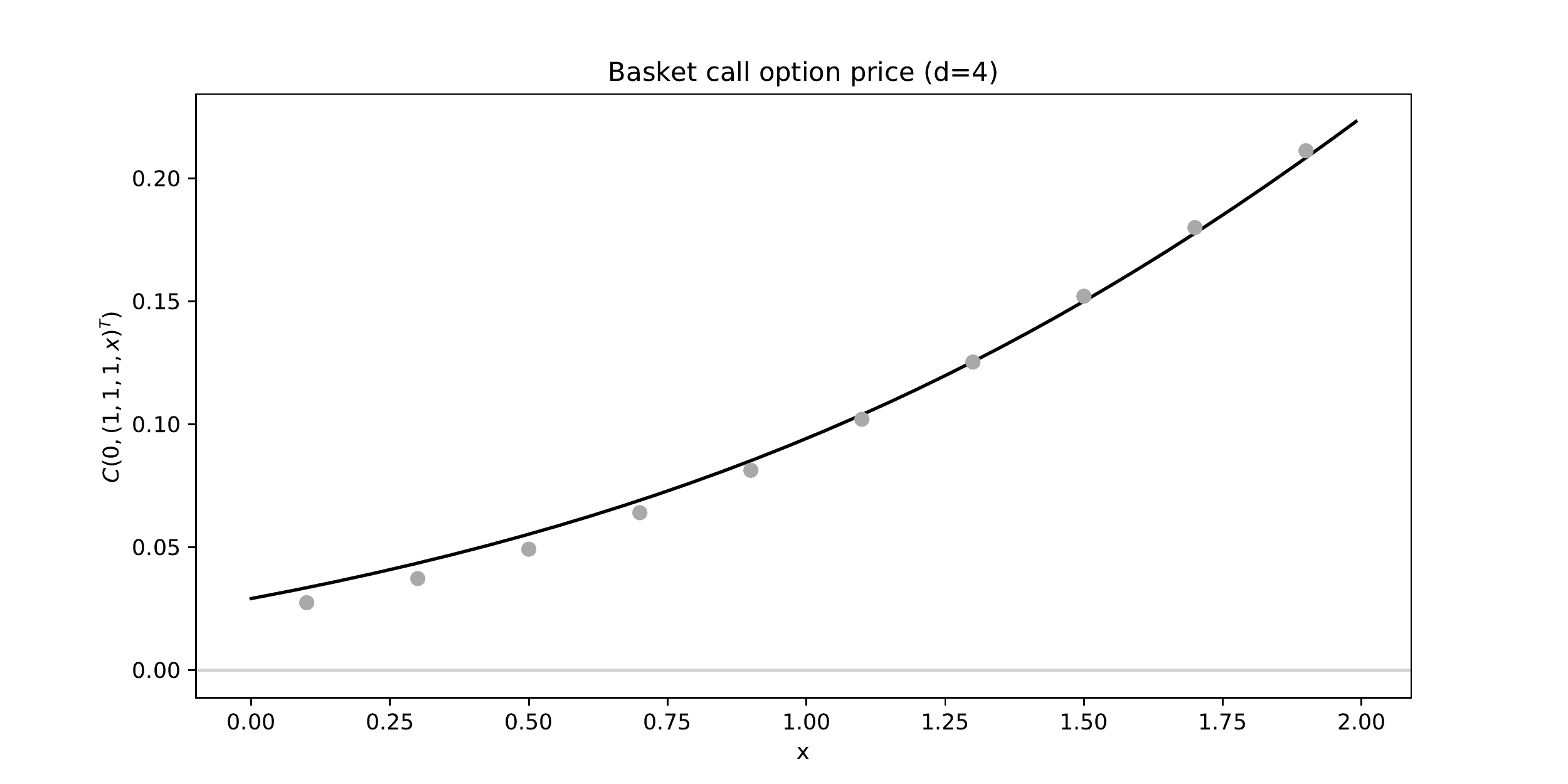}}%
	\qquad
	\subfloat[Loss function for $d=4$.]{\includegraphics[scale=0.32]{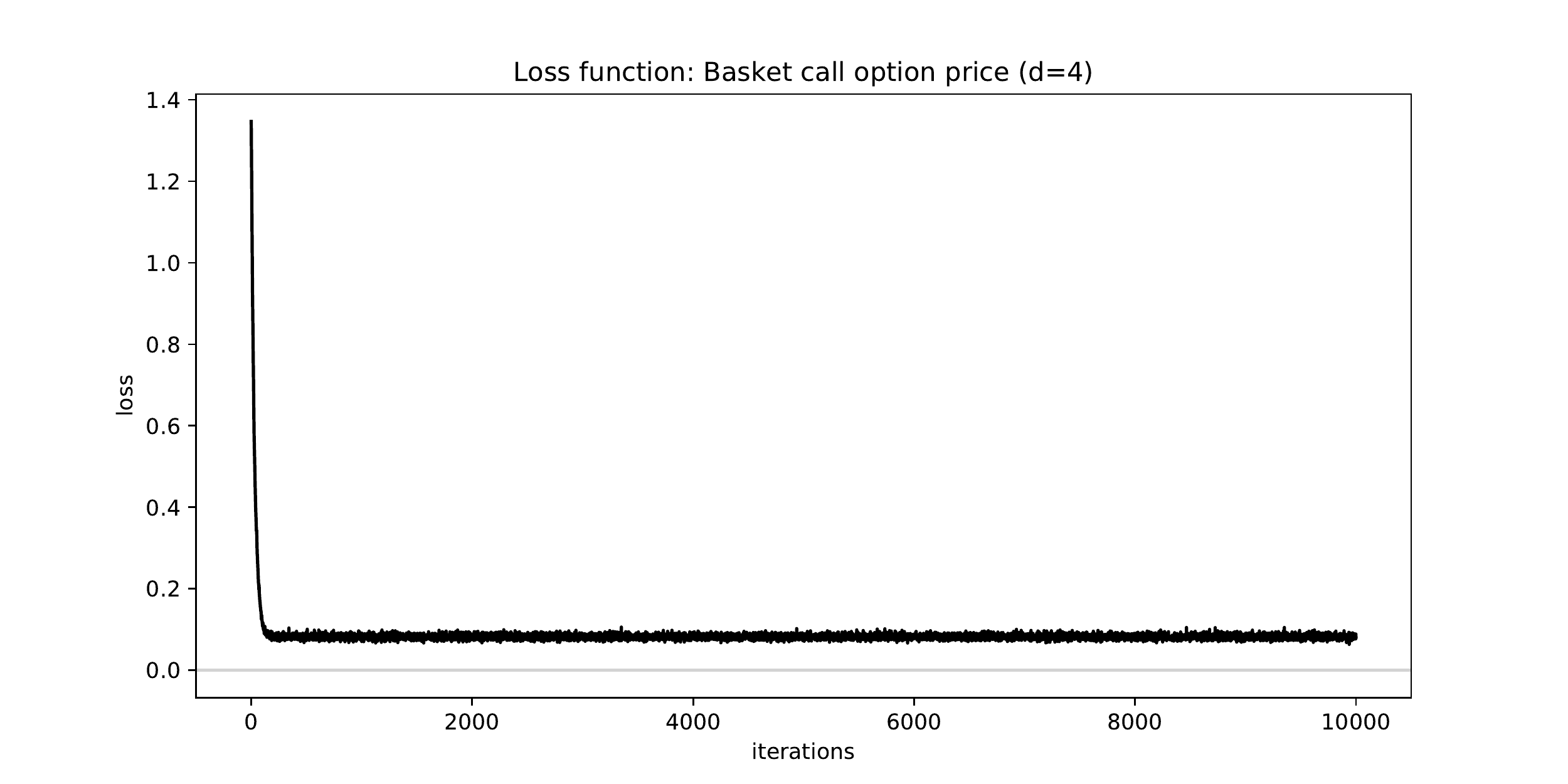}}%
	
	\subfloat[Estimate of $u(0,s)$ for $d=10$. ]{\includegraphics[scale=0.32]{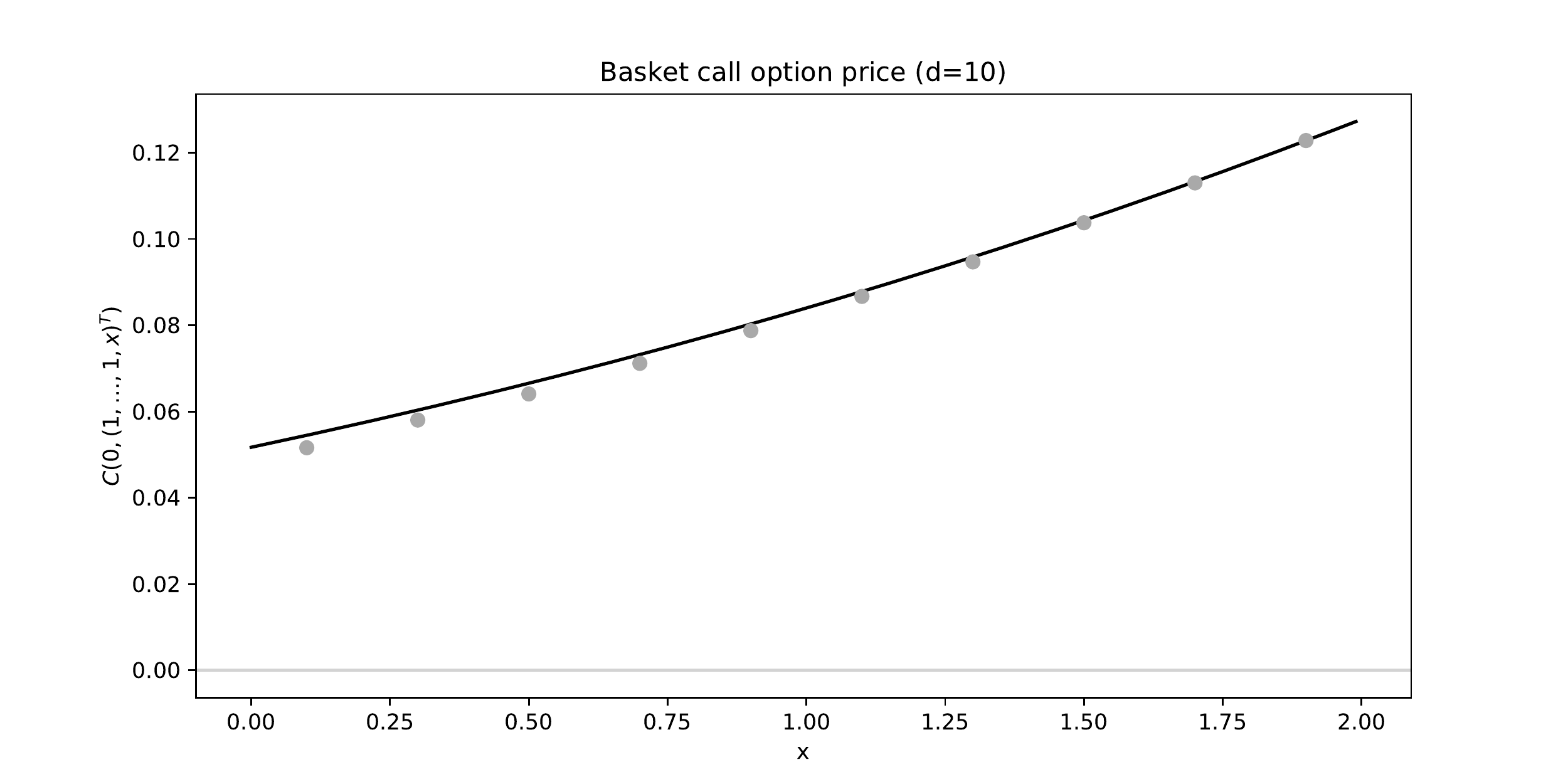}}%
	\qquad
	\subfloat[Loss function for $d=10$.]{\includegraphics[scale=0.32]{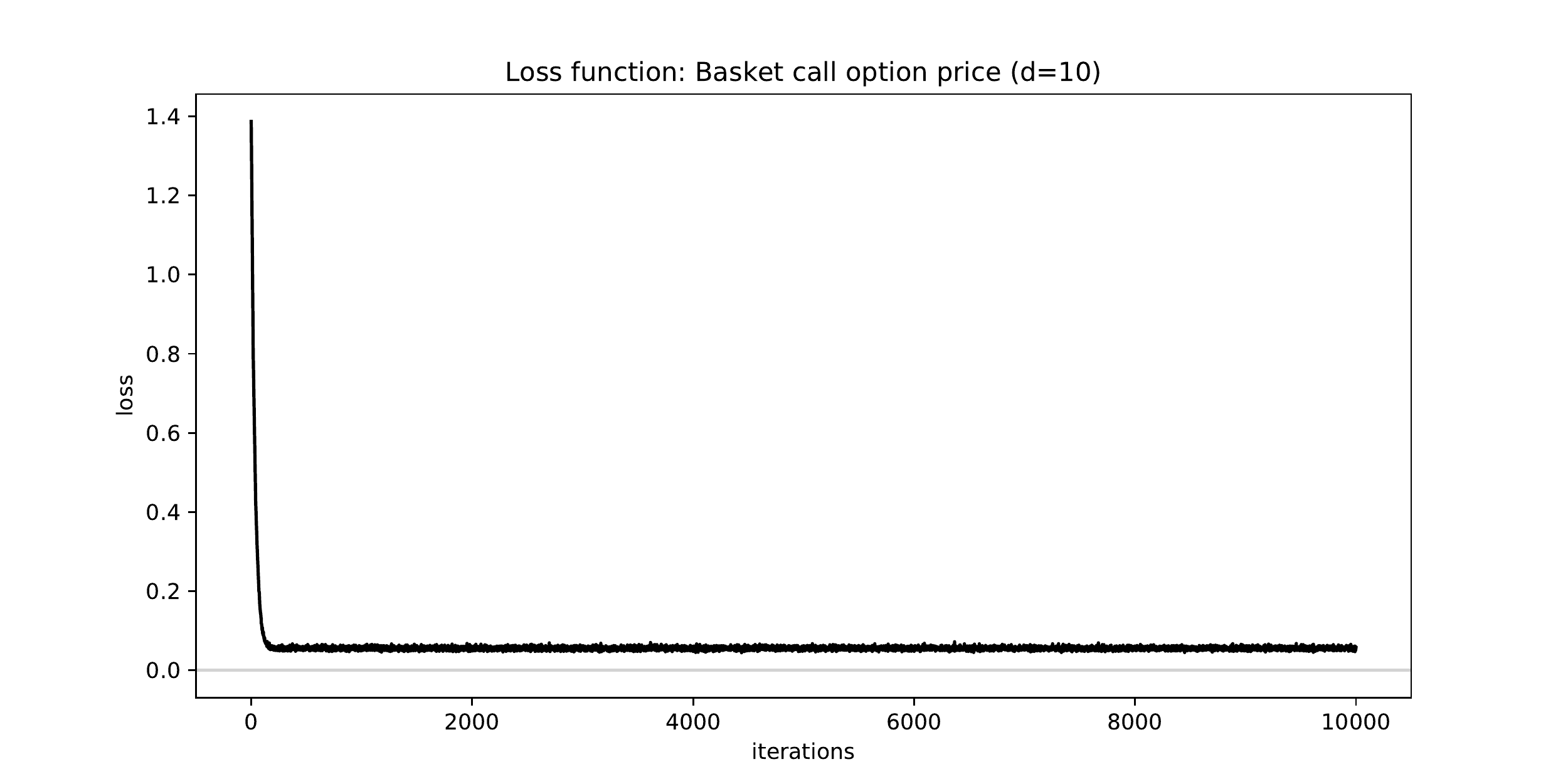}}%
	
	\caption{Estimates for the basket call option with 4 resp. 10 underlying assets. (A) and (B) show the estimation $\widehat {\mathcal U}_0(x)$ (black line) and MC-solutions (grey points). (B) and (C) show the loss function during the training procedure.}%
	\label{fig:basketcalloption}%
\end{figure}

\subsection{Semilinear PIDE}

In this section we study the semilinear \emph{stochastic linear regulator} problem for $T=1$. It is well known that there exists an analytical solution (see e.g. \citet{bib:oksendal2005stochastic}), which we use to verify the numerical estimation.

We train a neural network for every time point $t_i$ of a partition $0=t_0<t_1<\dots<t_N=T$ using batch size 10000 and 12000 gradient descent iterations with a learning rate 0.1 that decreases after 3000, 6000 and 9000 iterations by factor $10^{\,{-}1}$. We apply \texttt{sigmoid} activation function in the hidden layers and linear activation function in the output layer.

Denote by $c =(c_{1,t}, \dots, c_{d,t})_{t\ge 0}$  an adapted  control strategy and consider the controlled $d$-dimensional process $S^c$ with dynamics
\begin{align}\label{eq:state-process-regulator}
dS_{i,t}^c=c_{i,t} \, \ud t + \sigma_i \,  \ud W_{i,t} + \int_{\R}z \, 	\widetilde{J_i}(\ud z,\ud t), \; 1 \le i \le d,  \qquad S^c_0=x \in \R^d\, .	
\end{align}
Here $W = (W_1, \dots, W_d)$ is a $d$-dimensional standard Brownian motion,  $\theta \in \R^d$, $\rho \in \R^d $, $\sigma_1, \dots, \sigma_d$ are positive  constants, $T >0$   and $\widetilde J_1, \dots, \widetilde{J}_d$ is the compensated jump measure of $d$ independent compound   Poisson processes with Gamma($\alpha_i,\beta_i$)-distributed jumps. Denote by $\mathcal{A}$ the set of all adapted $d$-dimensional processes $c$ with $\bE \Big [ \int_0^T |c_s|^2  \,  \ud s\Big] < \infty $  and  consider the control problem
\begin{align*}
u(t,x)=\inf_{c \in \mathcal{A}} \bE \bigg [ \sum_{i=1}^d \Big( \int_{t}^{T} \big( (S_{i,s}^c)^2 + \theta_i c_{i,s}^2 \big) \,\ud s + \rho_i (S_{i,T}^c )^2 \Big) \,\Big |\, S_t^c=x \bigg ], \quad t \in [0,T], \, x \in \R^d.
\end{align*}
The interpretation of this problem is that the controller wants to drive  the process $S^c$ to zero using the control $c$; the instantaneous control cost  (for instance the energy consumed) is  measured by $\theta  c^2_t$. At maturity $T$ the controller incurs the terminal cost $\rho (S^c_T)^2$.

The Hamilton-Jacobi-Bellman (HJB) equation associated to this control problem  is
\begin{align*}
u_t(t,x)&+\frac{1}{2} \sum_{i=1}^{d} \sigma_i^2  u_{x_i x_i}(t,x) + \int_{\R^d} \Big [ u(t,x+z)-u(t,x)-\sum_{i=1}^d z_i u_{ x_i}(t,x) \Big ]  \,  \nu(\ud z)  \\ &+ \sum_{i=1}^d x_i^2 +\inf_c \Big \{\sum_{i=1}^d c_i u_{ x_i}(t,x) +\theta_i c_i^2  \Big \}=0,\quad (t,x) \in [0,T) \times \R^d,
\end{align*}
with terminal condition  $u(T,x)=\varphi(x):=\sum_{i=1}^d \rho_i x_i^2$.
The minimum  in the HJB equation is attained at
$c_{i}^*(t,x)=-\frac{1}{2 \theta_i} \frac{\partial u}{\partial x_i} (t,x)$, so that the value function solves the semilinear PIDE
\begin{align} \label{eq:HJB-stoch-reg}
u_t(t,x)&+\frac{1}{2} \sum_{i=1}^{d} \sigma_i^2  u_{x_i x_i}(t,x) - \sum_{i=1}^d   \int_{\R^d} z_i \, \nu(\ud z)   u_{x_i} (t,x)\\ &+\int_{\R^d}  [u(t,x+z)-u(t,x) ] \, \nu(\ud z)  + \sum_{i=1}^d x_i^2 - \sum_{i=1}^d \frac{1}{4 \theta_i} u_{ x_i}(t,x)^2=0\,.
\end{align}
It is well known that the HJB equation \eqref{eq:HJB-stoch-reg} can be solved analytically, see \cite{bib:oksendal2005stochastic}. For this we  make the Ansatz
$
u(t,x)=\sum_{i=1}^d a^i(t) x_i^2 +b(t).
$
Substitution into \eqref{eq:HJB-stoch-reg} gives an ODE system for $a(t)$ and $b(t)$ that can be solved explicitly. One obtains
\begin{align*}
a^i(t)&=\sqrt{\theta_i} \frac{1+\kappa_i e^{2t/\sqrt{\theta}}}{1-\kappa_i e^{2t/\sqrt{\theta}}}, \qquad \kappa_i:= \frac{\rho_i-\sqrt{\theta_i}}{\rho_i+\sqrt{\theta_i}}e^{-\frac{2T}{\sqrt{\theta_i}}} \\
b(t)&=\sum_{i=1}^d \sqrt{\theta_i} \Big(\sigma_i^2 + \int_{\R^d} z_i^2 \, \nu(\ud z)\Big)\Big((T-t)+\log \big((1-\kappa_i e^{2t})/(1-\kappa_i e^{2T})\big)\Big).
\end{align*}

To test the deep splitting method we  compute estimates for $u(t,x)$ for $x \in A: =  [-2,2] ^d$ for $d=4$ and $d=10$ with parameters
\begin{align*}
\begin{cases}
\sigma= (0.1,\dots,0.1)^T, \, \rho=  (1,\dots,1)^T \, , \theta=  (0.5,\dots,0.5)^T,\\
\lambda=(10,\dots,10)^T, \,  \alpha=(0.4,\dots,0.4)^T, \, \beta=(4,\dots,4)^T\, .	
\end{cases}
\end{align*}

For this we partition the time horizon into $N=10$ intervals $0=t_0<t_1<\dots<t_N=T$  and simulate the auxiliary process $\mathcal X$ for $t \in [t_{n-1},t_n]$
\begin{align} \label{eq:X}
\mathcal X_{i,t}=\xi_i  + \int_{t_{n-1}}^{t} \sigma_i \, \ud W_{i,s} + \int_{t_{n-1}}^{t}  \int_{\R} z \, \widetilde J_i(\ud z,\ud s) , \quad 1 \le i \le d,
\end{align}
where  $W$ and $\widetilde J$ are as in \eqref{eq:state-process-regulator}, and where $\xi$ is uniformly distributed on $A$.  The nonlinear term is finally given by
\begin{align*}
f(t,x,y,z)=\sum_{i=1}^d \Big(  \frac{1}{4\theta_i}z_i^2 -x_i^2 \Big).
\end{align*}
We linearize the PIDE and approximate $x \mapsto u(t_n,x)$ with a deep neural network $\widehat{\mathcal{U}}_{n}$ for $n=0,1,\dots,N-1$ by minimizing the loss function $L_n(\mathcal{U}_n)$ defined by \eqref{eq:loss_fun}.

The left column of Figure \ref{fig:stochreg} shows the estimated solutions and the analytic reference solutions $u(0,s)$ for $s=(x,1,\dots,1), x \in [-2,2]$ with $4$ resp. $10$ underlying state processes. In the right column is the corresponding loss function for the network at time $0$. Table \ref{tab:stochreg} reports the average estimate of $u(0,s_0)$ for $s_0=(1,\dots,1)^T$ and the standard deviation observed over 10 independent runs, the theoretical solutions and the relative error of the average estimate of $u(0,s_0)$. Moreover, as the true solution $u$ is known explicitly, we report the average error for 10 independent runs on the whole domain $[-2,2]^d$, which is (for each run) $\frac{1}{M}\sum_{i=1}^{M} \Big |\frac{\widehat{\mathcal{U}}_0(\xi_i)-u(0,\xi_i)}{u(0,\xi_i)} \Big |$ for iid $\xi_1,\dots,\xi_M \sim \text{Unif}([-2,2]^d)$  and $M=10000$.  For $d=4$ and $d=10$ one run takes approximately 3700 seconds resp. 5500 seconds.

\begin{table}[t]
	\begin{center}
		\begin{tabular}{l c c c c c }
			& Avg. value & Std. deviation & Theor. solution & Rel. err. (\%) & Rel. err. on $[-2,2]^d$ (\%) \\
			\hline
			\textbf{$d=4$} 	& 4.686643 &  0.052794	& 4.743960 & 1.21 & 1.97   \\
			\textbf{$d=10$} & 12.321216	 & 0.141565    & 11.859899	 & 3.89 & 2.62 \\
			\hline
		\end{tabular}
	\end{center}
	\vspace{2mm}
	\caption {Average estimates, standard deviations over 10 independent runs and theoretical solutions which are used as benchmark for the relative error for $u(0,s_0)$ for $s_0=(1,\dots,1)^T$, and the average error on the whole domain $[-2,2]^d$ is reported.}   \label{tab:stochreg}
\end{table}

\begin{figure} %
	\centering
	\subfloat[Estimate and theor. sol. $u(0,s)$ ($d=4$). ]{\includegraphics[scale=0.32]{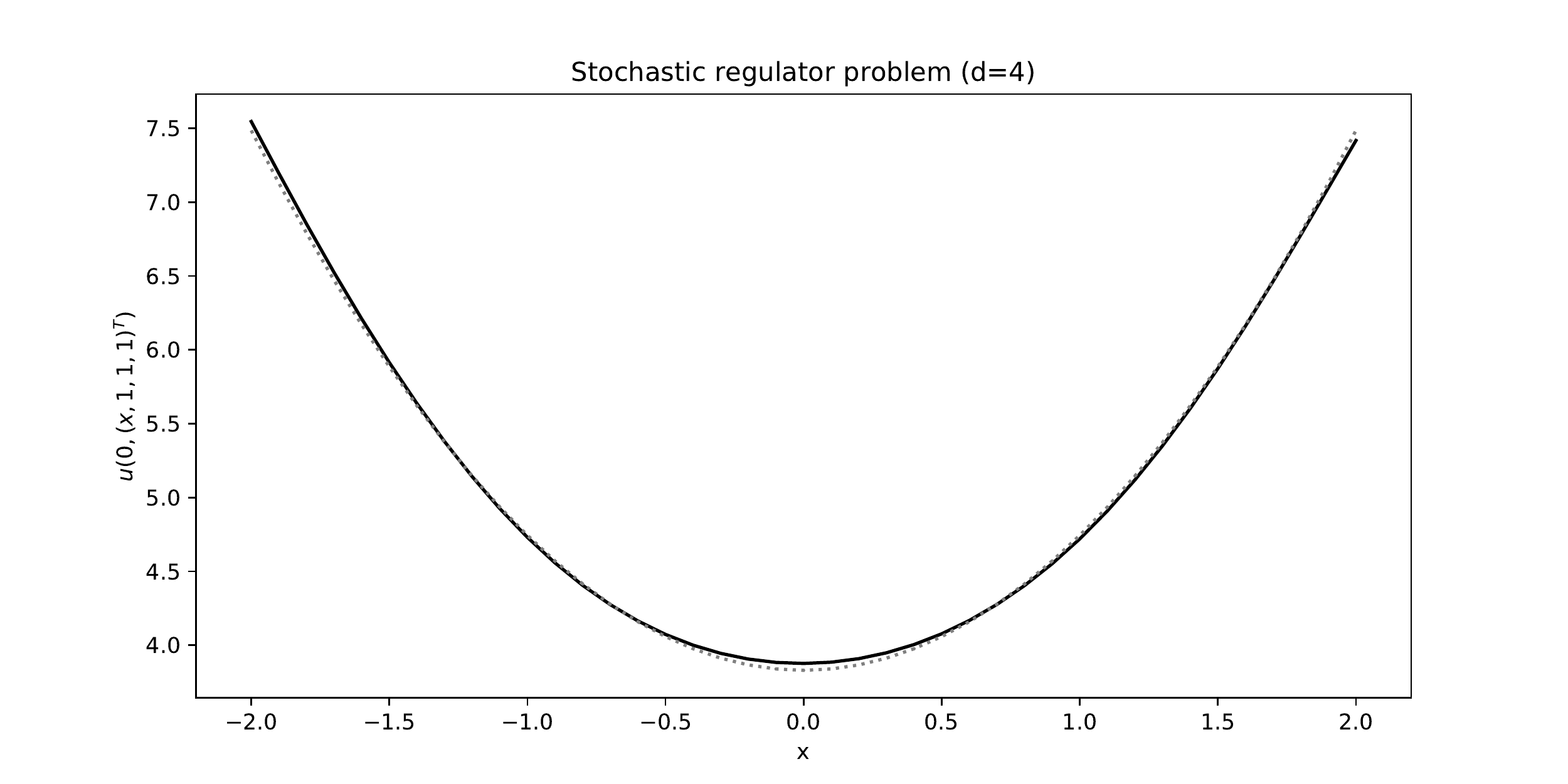}}%
	\qquad
	\subfloat[Loss function for $\widehat{ \mathcal U}_0$,  $d=4$.]{\includegraphics[scale=0.32]{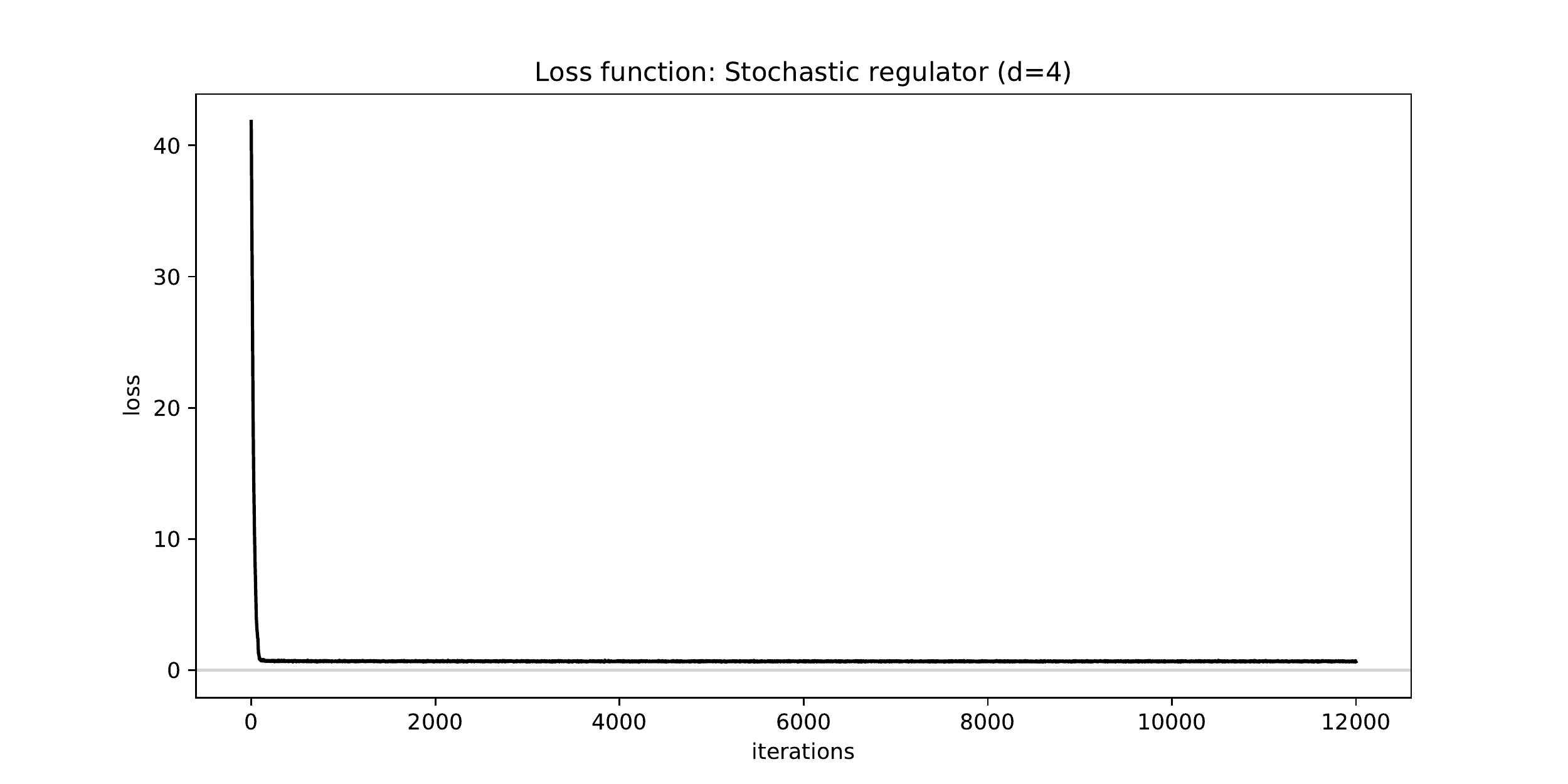}}%
	
	\subfloat[Estimate and theor. sol. $u(0,s)$ ($d=10$).  ]{\includegraphics[scale=0.32]{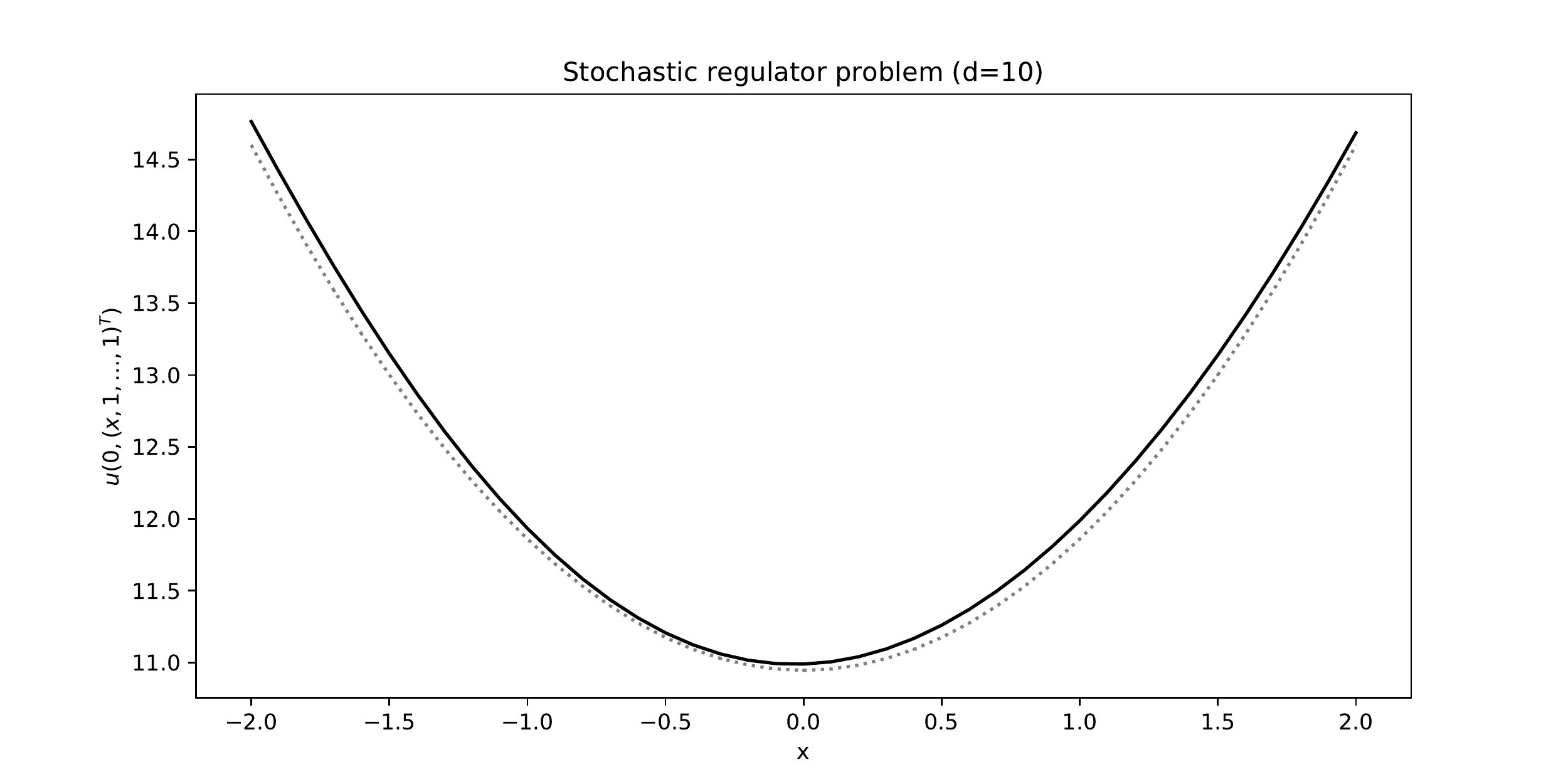}}%
	\qquad
	\subfloat[Loss function for $\widehat{ \mathcal U}_0$,  $d=10$.]{\includegraphics[scale=0.32]{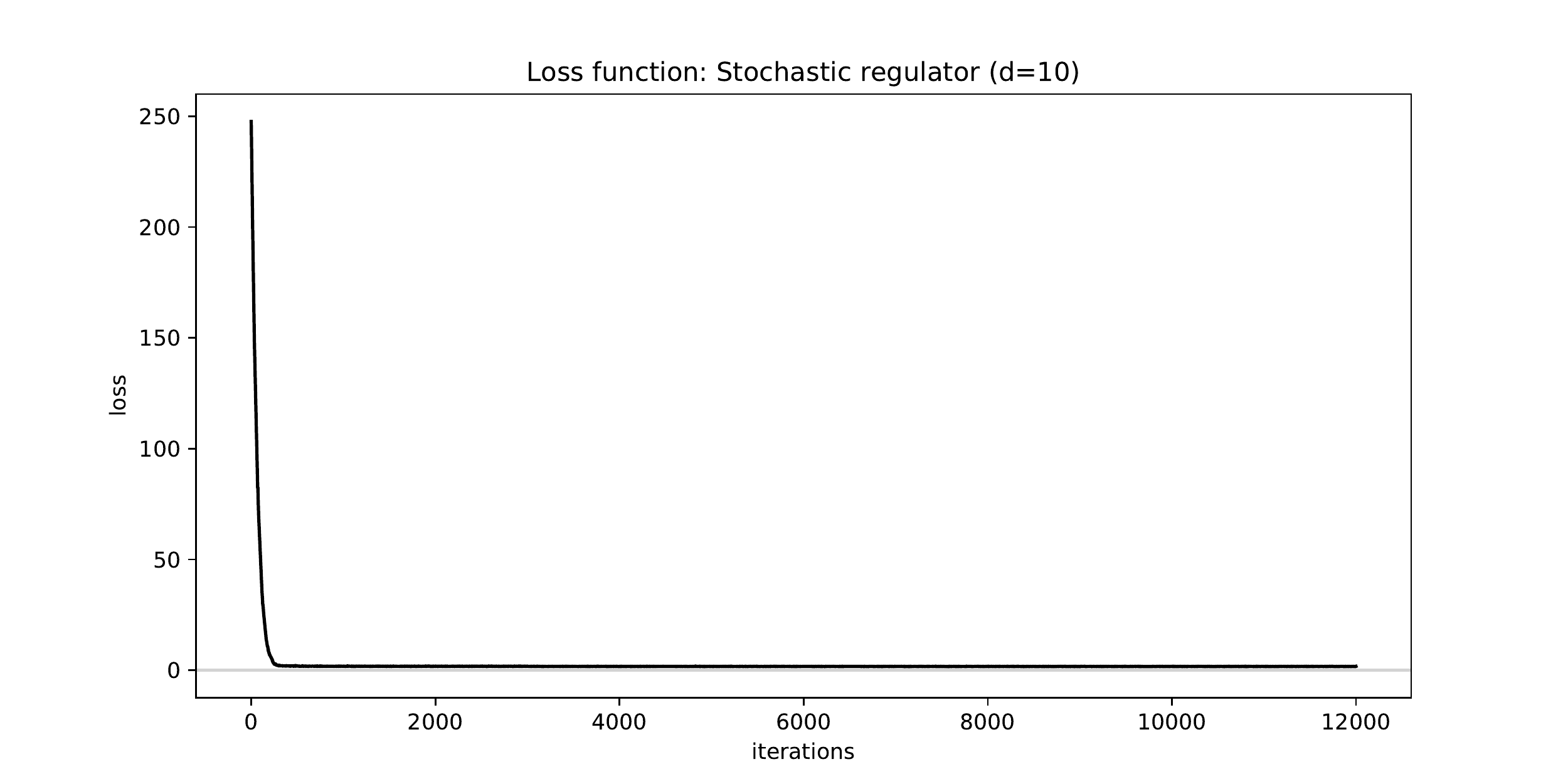}}%
	
	\caption{Estimates for the stochastic regulator problem with 4 resp. 10 underlying state processes. (A) and (C) shows the true solution (grey dotted line) and the deep neural network estimation (black solid line) for $u(0,s)$. (B) and (C) are the loss functions during the training procedure of $\widehat{\mathcal{U}}_0$. }%
	\label{fig:stochreg}%
\end{figure}

	\bibliographystyle{abbrvnat} 
\bibliography{pide_conv_bib}

\appendix

\section{Useful results} \label{sec:prel}

\subsection{Stochastic Calculus}

For details on stochastic calculus with L\'evy processes refer e.g. to \citet{applebaum2009levy} or \citet{delong2013backward}. We consider a probability space $(\Omega,\F,\P)$ with a filtration $\F=(\F_t)_{0 \le t \le T}$ that supports a Brownian motion $W$ and a Poisson random measure $J$. Recall some important results that we will frequently use in the following.

\begin{lemma}[Martingale Representation Theorem] \label{lemma:martingalerep}
	For any martingale $M$ there exists a $(Z,U) \in L^2_W(\R^d) \times L^2_J(\R)$ such that for $t \in [0,T]$,
	$$ M_t=M_0 + \int_{0}^{t} Z_s \ud W_s + \int_{0}^{t} U(s,z) \widetilde{J}(\ud s, \ud z) \, .$$
\end{lemma}

\begin{lemma}[Conditional It\^o Isometry] \label{lemma:isometry}
	For $V^1, V^2 \in L_J^2(\R)$ and $H, K \in L^2_W(\R^d)$ it holds that
	\begin{align*}
	\mathbb{E}_i \Big[ \int_{t}^{t_{i+1}}H_r \ud W_r \int_{t}^{t_{i+1}}K_r \ud W_r \Big]&=\mathbb{E}_i \Big[ \int_{t}^{t_{i+1}}H_r K_r \ud r \Big], \\
	\mathbb{E}_i \Big[ \int_{t}^{t_{i+1}} \int_{\R^d} V^1(s,z) \widetilde{J}(\ud s, \ud z) \int_{t}^{t_{i+1}} \int_{\R^d} V^2(s,z) \widetilde{J}(\ud s, \ud z) \Big]&=\mathbb{E}_i \Big[ \int_{t}^{t_{i+1}} \int_{\R^d} V^1(s,z) V^2(s,z) \nu(\ud z) \ud s \Big], \\
	\mathbb{E}_i \Big[ \int_{t}^{t_{i+1}} \int_{\R^d} V^1(s,z) \widetilde{J}(\ud s, \ud z) \int_{t}^{t_{i+1}}H_r \ud W_r\Big]& =0 \, .
	\end{align*}
\end{lemma}

\subsection{Inequalities}
We will often use the following classical inequalities that we recall for the convenience of the reader.
\subsubsection*{{Young inequality.}} For all $(a,b) \in \R^2, \beta>0$,
$$ (1-\beta)a^2+(1-\frac{1}{\beta})b^2 \le (a+b)^2 \le (1+\beta)a^2+(1+\frac{1}{\beta})b^2$$
\subsubsection*{Discrete Gronwall Lemma.} Let $(u_n,v_n,h_n)_{n}$ be positive sequences satisfying for all $n \in \bN$
$$u_n \le (1+h_n)u_{n+1}+v_n \, , \qquad \forall n \in \bN \, . $$
Then, we have for all $N \in \bN^*$
$$\sup_{i \in \{0,1,\dots,N\}} u_i \le \exp \big(  \sum_{i=1}^{N-1} h_i \big)\big(u_N+\sum_{i=1}^{N-1}v_i\big)$$
In particular when $h_i=\beta \Delta t$ with $\beta>0, \Delta t = O(1/N)$, there exists a $C>0$ independent of $N$ such that
$$\sup_{i \in \{0,1,\dots,N\}} u_i \le C\big(u_N+\sum_{i=1}^{N-1}v_i\big)$$

\end{document}